\documentclass[reqno]{amsart}
\usepackage{amssymb}
\usepackage{graphicx}

\usepackage[usenames, dvipsnames]{color}
\usepackage{verbatim}
\usepackage{mathrsfs}
\usepackage{bm}
\usepackage{cite}
\usepackage{esint}

\numberwithin{equation}{section}

\newtheorem{theorem}{Theorem}[section]

\newtheorem{lemma}[theorem]{Lemma}
\newtheorem{prop}[theorem]{Proposition}

\theoremstyle{definition}
\newtheorem{remark}[theorem]{Remark}

\theoremstyle{definition}
\newtheorem{definition}[theorem]{Definition}

\theoremstyle{definition}

\makeatletter
\def\dashint{\operatorname%
{\,\,\text{\bf-}\kern-.98em\DOTSI\intop\ilimits@\!\!}}
\makeatother

\def\\det{\text{det}}

\def\.5{\frac{1}{2}}

\newcommand{\RN}[1]{%
  \textup{\uppercase\expandafter{\romannumeral#1}}%
}

\renewcommand{\epsilon}{\varepsilon}

\newcounter{marnote}

\begin{document}
\title[Gradient estimates for the insulated conductivity problem]{Gradient estimates for the insulated conductivity problem with inclusions of the general $m$-convex shapes}


\author[Z.W. Zhao]{Zhiwen Zhao}

\address[Z.W. Zhao]{Beijing Computational Science Research Center, Beijing 100193, China.}
\email{zwzhao365@163.com}


\date{\today} 



\begin{abstract}
In this paper, the insulated conductivity model with two touching or close-to-touching inclusions is considered in $\mathbb{R}^{d}$ with $d\geq3$. We establish the pointwise upper bounds on the gradient of the solution for the generalized $m$-convex inclusions under these two cases with $m\geq2$, which show that the singular behavior of the gradient in the thin gap between two inclusions is described by the first non-zero eigenvalue of an elliptic operator of divergence form on $\mathbb{S}^{d-2}$. Finally, the sharpness of the estimates is also proved for two touching axisymmetric insulators, especially including curvilinear cubes.
\end{abstract}

\maketitle



\section{Introduction}

Assume that $D\subseteq\mathbb{R}^{d}\,(d\geq3)$ is a bounded open set with $C^{2}$ boundary, whose interior embodies two adjacent $C^{2,\gamma}$-subdomains $D_{1}$ and $D_{2}$ with $0<\gamma<1$. Denote $\varepsilon:=\mathrm{dist}(D_{1},D_{2})$, where $\varepsilon\geq0$. Especially when $\varepsilon=0$, it means that $D_{1}$ and $D_{2}$ touch only at one point. Suppose also that $D_{i}$, $i=1,2,$ stay far away from the external boundary $\partial D$. Write $\Omega:=D\setminus\overline{D_{1}\cup D_{2}}$. In this paper, for given boundary data $\varphi\in C^{2}(\partial D)$, we aim to study the singular behavior of the gradient of a solution to the insulated conductivity problem with $C^{\gamma}$ coefficients as follows:
\begin{align}\label{con002}
\begin{cases}
-\partial_{i}(A_{ij}(x)\partial_{j}u)=0,&\hbox{in}\;\Omega,\\
A_{ij}(x)\partial_{j}u(x)\nu_{i}=0,&\mathrm{on}\;\partial D_{i},\,i=1,2,\\
u=\varphi, &\mathrm{on}\;\partial D,
\end{cases}
\end{align}
where the coefficient matrix $(A_{ij}(x))\in C^{\gamma}$ is symmetric and verifies $\varsigma I\leq A(x)\leq\frac{1}{\varsigma}I$ for some positive constant $\varsigma$, $\nu$ is the unit outer normal to the subdomains $D_{1}$ and $D_{2}$.

To state our main results in a precise manner, we first formulate the domain. By suitable translation and rotation of the coordinates, we have
\begin{align*}
D_{1}:=D_{1}^{\ast}+(0',\varepsilon/2),\quad\mathrm{and}\; D_{2}:=D_{2}^{\ast}+(0',-\varepsilon/2),
\end{align*}
where $D_{1}^{\ast}$ and $D_{2}^{\ast}$ are touching only at the origin and satisfy
\begin{align*}
D_{i}^{\ast}\subset\{(x',x_{d})\in\mathbb{R}^{d}\,|\,(-1)^{i+1}x_{d}>0\},\quad i=1,2.
\end{align*}
Here and throughout the paper, we denote $(d-1)$-dimensional variables and domains by adding superscript prime, for instance, $x'$ and $B'$.

Assume further that there exists a small $\varepsilon$-independent constant $R_{0}>0$ such that the portions of $\partial D_{1}$ and $\partial D_{2}$ around the origin are, respectively, the graphs of two $C^{2,\gamma}$ functions $\varepsilon/2+h_{1}(x')$ and $-\varepsilon/2+h_{2}(x')$, where $h_{j}$, $j=1,2$ satisfy the following $m$-convex conditions: for $m\geq2$ and $\gamma>0$, $x'\in B'_{2R_{0}}$,
\begin{enumerate}
{\it\item[(\bf{H1})]
$h_{1}(x')-h_{2}(x')=\kappa_{0}\Big(\sum\limits_{i\in\mathcal{A}}\kappa_{i}|x_{i}|^{2}\Big)^{\frac{m}{2}}+\sum\limits_{j\in\mathcal{B}}\kappa_{j}|x_{j}|^{m}+O(|x'|^{m+\gamma})$,
\item[(\bf{H2})]
$|\nabla_{x'}h_{j}(x')|\leq \tau_{1}|x'|^{m-1},$\;$j=1,2,$
\item[(\bf{H3})]
$\|h_{1}\|_{C^{2}(B'_{2R_{0}})}+\|h_{2}\|_{C^{2}(B'_{2R_{0}})}\leq \tau_{2},$}
\end{enumerate}
where $\mathcal{A}$ and $\mathcal{B}$ are two sets such that $\mathcal{A}\cup\mathcal{B}=\{1,...,d-1\}$ and $\mathcal{A}\cap\mathcal{B}=\emptyset,$  $\kappa_{i}$, $i=0,1,...,d-1$, $\tau_{1}$ and $\tau_{2}$ are all positive constants independent of $\varepsilon$. Here and below, the notation $O(A)$ represents that there exists a $\varepsilon$-independent positive constant $C$ such that $|O(A)|\leq CA$. We should point out that the case of $m=2$ in condition $\mathrm{(}${\bf{H1}}$\mathrm{)}$ corresponds to the strictly convex inclusions, that is, the principal curvatures of interfacial boundaries of inclusions are greater than zero, which has been studied in \cite{DLY2022}. The results in \cite{DLY2022} revealed that the gradient blow-up rate depends on the principal curvatures of the surfaces of insulators, which is different from the blow-up phenomenon occurring in the perfect conductivity problem. However, when $m>2$, the principal curvatures degenerate to be zero and the surfaces of inclusions become flatter. In this case, the shapes of inclusions considered in condition $\mathrm{(}${\bf{H1}}$\mathrm{)}$ are formed by the coupling of two different types of $m$-convex curved surfaces as follows: $\kappa_{0}\big(\sum\limits_{i\in\mathcal{A}}\kappa_{i}|x_{i}|^{2}\big)^{\frac{m}{2}}+O(|x'|^{m+\gamma})$ and $\sum\limits_{j\in\mathcal{B}}\kappa_{j}|x_{j}|^{m}+O(|x'|^{m+\gamma})$. This complex $m$-convex structure will increase the difficulties of analysis and computations, which leads to that the generalization from $m=2$ to $m>2$ is not trivial. In particular, it requires more complex but precise calculations to achieve this generalization. These required changes will be mainly embodied with the following proof procedures for Theorem \ref{thm001}, especially in pages 15--21 below.

We additionally emphasize that the shapes of insulators under condition $\mathrm{(}${\bf{H1}}$\mathrm{)}$ contain curvilinear cubes as follows: the interfacial boundaries of $D_{1}$ and $D_{2}$ are, respectively, formulated as
\begin{align}\label{CUBES001}
\sum^{d-1}_{i=1}|x_{i}|^{m}+|x_{d}-\varepsilon/2-r_{1}|^{m}=r_{1}^{m},\quad\sum^{d-1}_{i=1}|x_{i}|^{m}+|x_{d}+\varepsilon/2+r_{2}|^{m}=r^{m}_{2},
\end{align}
where $r_{i}>0$, $i=1,2,$ are independent of $\varepsilon$. In fact, from Taylor expansion, we have
\begin{align*}
h_{1}(x')-h_{2}(x')=\bar{\kappa}\sum^{d-1}_{i=1}|x_{i}|^{m}+O(|x'|^{2m}),\quad\mathrm{in}\;\Omega_{r_{0}},
\end{align*}
where $\bar{\kappa}=m^{-1}(r_{1}^{1-m}+r_{2}^{1-m})$ and $0<r_{0}<\min\{r_{1},r_{2}\}$. So curvilinear cube belongs to the case when $\mathcal{A}=\emptyset$ and $\kappa_{i}=\bar{\kappa}$, $i\in\mathcal{B}$ in condition $\mathrm{(}${\bf{H1}}$\mathrm{)}$. Moreover, this type of axisymmetric inclusions has been widely used in the manufacture of composite materials due to its regular shape and fine properties.

For $y'\in B'_{R_{0}},\,0<s\leq2R_{0}$, denote a thin gap by
\begin{align*}
\Omega_{s}(y'):=&\{x\in \mathbb{R}^{d}\,|\,-\varepsilon/2+h_{2}(x')<x_{d}<\varepsilon/2+h_{1}(x'),~|x'-y'|<s\},
\end{align*}
whose upper and lower boundaries are, respectively, written by
\begin{align*}
\Gamma^{+}_{s}:=\{x\in\mathbb{R}^{d}\,|\,x_{d}=\varepsilon/2+h_{1}(x'),\;|x'|<s\},
\end{align*}
and
\begin{align*}
\Gamma^{-}_{s}:=\{x\in\mathbb{R}^{d}\,|\,x_{d}=-\varepsilon/2+h_{2}(x'),\;|x'|<s\}.
\end{align*}
For simplicity, write $\Omega_{t}(0')$ as $\Omega_{t}$ in the case of $y'=0'$. Observe that from the standard elliptic estimates, we have $\|u\|_{C^{1}(\Omega\setminus\Omega_{R_{0}/2})}\leq C.$ Therefore, it suffices to study the insulated conductivity problem in a narrow region as follows:
\begin{align}\label{problem006}
\begin{cases}
-\partial_{i}(A_{ij}(x)\partial_{j}u)=0,&\hbox{in}\;\Omega_{R_{0}},\\
A_{ij}(x)\partial_{j}u(x)\nu_{i}=0,&\mathrm{on}\;\Gamma^{\pm}_{R_{0}},\\
\|u\|_{L^{\infty}(\Omega_{R_{0}})}\leq1.
\end{cases}
\end{align}

Consider the following eigenvalue problem:
\begin{align}\label{eigen001}
-\mathrm{div}_{\mathbb{S}^{d-2}}(\kappa(\xi)\nabla_{\mathbb{S}^{d-2}}u(\xi))=\lambda\kappa(\xi)u(\xi),\quad\xi\in\mathbb{S}^{d-2},
\end{align}
where $\kappa(\xi)=\kappa_{0}(\sum_{i\in\mathcal{A}}\kappa_{i}|\xi_{i}|^{2})^{m/2}+\sum_{j\in\mathcal{B}}\kappa_{j}|\xi_{j}|^{m}$ satisfies that $\|\ln\kappa\|_{L^{\infty}(\mathbb{S}^{d-2})}<\infty$. Define the inner product as follows:
\begin{align}\label{inner001}
\langle u,v\rangle_{\mathbb{S}^{d-2}}=\fint_{\mathbb{S}^{d-2}}\kappa(\xi)uv.
\end{align}
By the classical eigenvalue theory of elliptic operator of divergence form on $\mathbb{S}^{d-2}$, we know that each eigenvalue of problem \eqref{eigen001} is real and the corresponding normalized eigenfunctions form an orthonormal basis of $L^{2}(\mathbb{S}^{d-2})$ under the inner-product \eqref{inner001}. Moreover, the first nonzero eigenvalue $\lambda_{1}$ of problem \eqref{eigen001} can be determined by the Rayleigh quotient:
\begin{align*}
\lambda_{1}=\inf\limits_{u\not\equiv0,\,\langle u,1\rangle_{\mathbb{S}^{d-2}}=0}\frac{\fint_{\mathbb{S}^{d-2}}\kappa(\xi)|\nabla_{\mathbb{S}^{d-2}}u|^{2}}{\fint_{\mathbb{S}^{d-2}}\kappa(\xi)|u|^{2}}.
\end{align*}
Denote
\begin{align}\label{degree}
\alpha(\lambda_{1}):=\frac{-(d+m-3)+\sqrt{(d+m-3)^{2}+4\lambda_{1}}}{2},
\end{align}

Unless otherwise stated, in the following the constant $C$ may change from line to line, which depends only on $d,m,\varsigma,\gamma,R_{0},\tau_{1},\tau_{2}$, $\kappa_{i}$, $i=0,1,...,d-1$, and $\|A\|_{C^{\gamma}}$,  but not on $\varepsilon$. First, we establish the pointwise upper bounds on the gradient as follows.
\begin{theorem}\label{thm001}
Suppose that $D_{1},\,D_{2}\subset D\subseteq\mathbb{R}^{d}\,(d\geq3)$ are defined as above, conditions $\mathrm{(}${\bf{H1}}$\mathrm{)}$--$\mathrm{(}${\bf{H3}}$\mathrm{)}$ hold. Let $u\in H^{1}(\Omega_{R_{0}})$ be the solution of \eqref{problem006} with $A_{ij}(0)=\delta_{ij}$. Then

$(i)$ if $\varepsilon=0$ and $x\in\Omega_{R_{0}/2}\setminus\{x'=0'\}$,
\begin{align}\label{ma001}
|\nabla u(x)|\leq C\|u\|_{L^{\infty}(\Omega_{R_{0}})}|x'|^{\alpha(\lambda_{1})-1};
\end{align}

$(ii)$ if $\varepsilon>0$ is sufficiently small and $x\in\Omega_{R_{0}/4}$,
\begin{align}\label{U002}
|\nabla u(x)|\leq C\|u\|_{L^{\infty}(\Omega_{R_{0}})}(\varepsilon+|x'|^{m})^{\frac{\alpha(\lambda_{1})-1}{m}},
\end{align}
where $\alpha(\lambda_{1})$ is defined by \eqref{degree}.
\end{theorem}
\begin{remark}
Using Lemma 5.1 in \cite{DLY2022}, we know that $\lambda_{1}\leq d-2$ and the equality holds if and only if $\kappa$ is constant. Then in the case of $m=2$, if condition $\mathrm{(}${\bf{H1}}$\mathrm{)}$ holds with $\mathcal{B}=\emptyset$ and $\kappa_{i}=\kappa_{j}$, $i,j\in\mathcal{A}$, or $\mathcal{B}=\{1,...,d-1\}$ and $\kappa_{i}=\kappa_{j}$, $i,j\in\mathcal{B}$, or $\mathcal{B}\neq\emptyset$, $\mathcal{B}\neq\{1,...,d-1\}$ and $\kappa_{0}\kappa_{i}=\kappa_{j}$, $i\in\mathcal{A}$, $j\in\mathcal{B}$, then $\lambda_{1}=d-2$, see \cite{DLY2021}. For $m>2$, if $\mathrm{(}${\bf{H1}}$\mathrm{)}$ holds with $\mathcal{B}\neq\emptyset$ or $\mathcal{B}=\emptyset$, $\kappa_{i_{1}}\neq\kappa_{i_{2}}$ for some $i_{1},i_{2}\in\mathcal{A}$, $i_{1}\neq i_{2}$, then $\lambda_{1}=d-2$, see \cite{Z2022}. Otherwise, we have $\lambda_{1}<d-2$.

\end{remark}

\begin{remark}
The results in Theorem \ref{thm001} can be extended to the case when $A_{ij}(0)\neq\delta_{ij}$ by making use of a appropriate linear transformation which reduces it to the case of $A_{ij}(0)=\delta_{ij}$, see Section 7 in \cite{DLY2022} for further details.

\end{remark}

For the purpose of establishing the optimal lower bound on the gradient, we assume that $D=B_{5}$ and $D_{1},\,D_{2}\subset B_{4}$ are two smooth $m$-convex inclusions such that $\overline{D}_{1}\cap\overline{D}_{2}=\{0\}$ and the domain $\Omega=D\setminus\overline{D_{1}\cup D_{2}}$ are symmetric with respect to every $x_{i}$, $i=1,2,...,d$. The optimal lower bound on the gradient is stated as follows.
\begin{theorem}\label{thm002}
For $d\geq3$, let $D,\,D_{1}$ and $D_{2}$ be described as above. Conditions $\mathrm{(}${\bf{H1}}$\mathrm{)}$--$\mathrm{(}${\bf{H3}}$\mathrm{)}$ hold with $\mathcal{A}=\emptyset$ and $\gamma>1-\alpha(\lambda_{1})$. Suppose that the eigenspace corresponding to the first nonzero eigenvalue $\lambda_{1}$ of \eqref{eigen001} contains a function which is odd with respect to some $x_{j_{0}}$, $j_{0}\in\{1,...,d-1\}$. Let $u\in H^{1}(\Omega)$ be the solution of \eqref{con002} with $\varepsilon=0$, $\varphi=x_{j_{0}}$ and $A_{ij}(x)\equiv\delta_{ij}$. Then
\begin{align*}
\limsup\limits_{x\in\Omega,\,|x|\rightarrow0}|x'|^{1-\alpha(\lambda_{1})}|\nabla u(x)|>\frac{1}{C},
\end{align*}
where $\alpha(\lambda_{1})$ is given by \eqref{degree} and the positive constant $C$ depends only on $d,m$, $\kappa_{i}$, $i=0,1,...,d-1$, and upper bounds of $\|\partial D_{j}\|_{C^{4}}$, $j=1,2.$
\end{theorem}
\begin{remark}
It is worth mentioning that the validity of the assumed condition ``the eigenspace corresponding to the first nonzero eigenvalue $\lambda_{1}$ of \eqref{eigen001} contains a function which is odd with respect to some $x_{j_{0}}$, $j_{0}\in\{1,...,d-1\}$" will be demonstrated in Section \ref{SEC03} below. In addition, the shape of inclusions considered in Theorem \ref{thm002} also contains curvilinear cubes with the same radii in \eqref{CUBES001}.

\end{remark}

The paper is organized as follows. The proofs of Theorems \ref{thm001} and \ref{thm002} are, respectively, given in Sections \ref{SEC02} and \ref{SEC03}. In the rest of the introduction we review some earlier relevant results.

The mathematical model for the conductivity problem can be described by the following elliptic equations of divergence form:
\begin{align}\label{pro006}
\begin{cases}
\mathrm{div}(a_{k}(x)\nabla u_{k})=0,&\mathrm{in}\;D,\\
u_{k}=\varphi,&\mathrm{on}\;\partial D,
\end{cases}\quad a_{k}(x)=&
\begin{cases}
k\in(0,\infty),&\mathrm{in}\;D_{1}\cup D_{2},\\
1,&\mathrm{in}\;\Omega,
\end{cases}
\end{align}
where $\varphi\in C^{2}(\partial D)$ is a given boundary data and $k$ is called the conductivity. When $k$ tends to zero, problem \eqref{pro006} becomes the insulated conductivity problem, while it turns into the perfect conductivity problem if $k$ goes to infinity. It is well known that there always appears blow-up of the gradient $|\nabla u|$ for the insulated or perfect conductivity problem, as the distance between two adjacent inclusions approaches to zero. Babu\u{s}ka et al. \cite{BASL1999} were the first to propose the problem of estimating $|\nabla u|$ in the close touching regime. In \cite{BASL1999} they analyzed computationally the damage and fracture of composites modeled by the Lam\'{e} system with finite coefficients and found that the gradient of solutions stays bounded independent of the distance between inclusions. For two touching disks, Bonnetier and Vogelius \cite{BV2000} proved the boundness for the conductivity problem \eqref{pro006}. Li and Vogelius \cite{LV2000} then extended their results to general second-order elliptic equations of divergence form with piecewise H\"{o}lder coefficients. In particular, their results hold for arbitrarily smooth shape of inclusions in all dimensions. The subsequent work \cite{LN2003} completed by Li and Nirenberg further extended to general second-order elliptic systems of divergence form, especially covering the Lam\'{e} system. This especially demonstrates the numerical observation in \cite{BASL1999}. For more related investigations on the elliptic equation with piecewise constant coefficients, see \cite{DL2019,DZ2016,JK2021,CEG2014,KL2019}.

Ammari et al.\cite{AKL2005,AKLLL2007} firstly made use of the layer potential techniques to establish the optimal gradient estimates for the insulated conductivity problem with two close-to-touching disks, which showed that the blow-up rate of the gradient is $\varepsilon^{-1/2}$ in two dimensions. For two adjacent spherical insulators in dimension three, Yun \cite{Y2016} obtained the optimal gradient estimates in the shortest segment between two inclusions and captured the blow-up rate of order $\varepsilon^{\frac{\sqrt{2}-2}{2}}$. For the general strictly convex inclusions, Bao, Li and Yin \cite{BLY2010} developed a ``flipping" technique to establish the pointwise upper bound estimate of the gradient in all dimensions as follows:
\begin{align}\label{LYL90}
|\nabla u|\leq C(\varepsilon+|x'|^{2})^{-1/2},\quad\text{in }\Omega_{R_{0}}.
\end{align}
Li and Yang \cite{LY2021,LY202102} further considered the general $m$-convex inclusions with $m\geq2$ and obtained
\begin{align}\label{DM001}
|\nabla u(x)|\leq C(\varepsilon+|x'|^{m})^{-1/m+\beta},\quad\mathrm{in}\;\Omega_{R_{0}},
\end{align}
for some inexplicit $\beta>0$. This improves the result in \eqref{LYL90}. Moreover, the upper bound \eqref{DM001} indicates that the blow-up rate $\varepsilon^{-1/m+\beta}$ will decrease as the convexity index $m$ increase. This also implies that curvilinear cubes are superior to spheres from the view of shape design of insulated materials. For the purpose of making clear the value of $\beta$ in \eqref{DM001}, Weinkove \cite{W2021} constructed a appropriate auxiliary function and used the maximum principle to solve an explicit constant $\beta(d)$ for two nearly touching balls in dimension greater than three. When condition $\mathrm{(}${\bf{H1}}$\mathrm{)}$ becomes $(h_{1}-h_{2})(x')=\kappa_{0}|x'|^{m}+O(|x'|^{m+\gamma})$ in $\Omega_{R_{0}}$, Dong, Li and Yang \cite{DLY2021} established the optimal upper and lower bounds on the gradient in the case of $m=2$ and $d\geq3$ and found that the optimal value of $\beta$ is $[-(d-1)+\sqrt{(d-1)^{2}+4(d-2)}]/4$. Zhao \cite{Z2022} further extended the results to the case when $m>2$ and $d\geq3$ and revealed that the optimal gradient blow-up rate is $\varepsilon^{-1/m+\beta(d,m)}$ with $\beta(d,m)=[-(d+m-3)+\sqrt{(d+m-3)^{2}+4(d-2)}]/(2m)$. Ma \cite{M2022} recently improved the auxiliary function constructed in \cite{W2021} and obtained the same $\beta$ as in \cite{DLY2021}.

With regard to the perfect conductivity problem, there has been a long list of literature making clear the blow-up phenomenon, for example, see \cite{AKLLL2007,BC1984,AKL2005,Y2007,Y2009,K1993,BLY2009,LY2009,BLY2010,L2012} and the references therein. For nonlinear equation, we refer to \cite{CS2019,CS201902,G2012} for an interested reader.

\section{The proof of Theorem \ref{thm001}}\label{SEC02}
For $\varepsilon\geq0$, define
\begin{align*}
\delta:=\delta(x')=:\varepsilon+\kappa\Big(\frac{x'}{|x'|}\Big)|x'|^{m},\quad |x'|\leq R_{0},
\end{align*}
where
\begin{align*}
\kappa\Big(\frac{x'}{|x'|}\Big)=\kappa_{0}\Big(\sum\limits_{i\in\mathcal{A}}\kappa_{i}|x_{i}|^{2}|x'|^{-2}\Big)^{\frac{m}{2}}+\sum\limits_{j\in\mathcal{B}}\kappa_{j}|x_{j}|^{m}|x'|^{-m}.
\end{align*}
For $\sigma,\tau\in\mathbb{R}$, define a norm as follows:
\begin{align*}
\|F\|_{\varepsilon,\gamma,\sigma,B_{R}'}:=\sup\limits_{x'\in B_{R}'}|x'|^{-\sigma}(\varepsilon+|x'|^{m})^{\tau-1}|F(x')|,\;\,\mathrm{with}\;0<R\leq R_{0}.
\end{align*}
Define the weighted space $H^{1}(B_{R_{0}}',|x'|^{m}dx')$ under a weighted norm as follows:
\begin{align*}
\|f\|_{H^{1}(B_{R_{0}}',|x'|^{m}dx')}:=\left(\int_{B_{R_{0}}'}|f|^{2}|x'|^{m}dx'\right)^{\frac{1}{2}}+\left(\int_{B_{R_{0}}'}|\nabla f|^{2}|x'|^{m}dx'\right)^{\frac{1}{2}}.
\end{align*}
For $0<\rho<R_{0}$, write
\begin{align*}
(f)^{\kappa}_{\partial B_{\rho}'}:=&\left(\int_{\partial B_{\rho}'}\kappa\Big(\frac{x'}{|x'|}\Big)dS\right)^{-1}\int_{\partial B_{\rho}'}\kappa\Big(\frac{x'}{|x'|}\Big)f(x')dS,\\
(f)^{\kappa}_{B_{\rho}'}:=&\left(\int_{B_{\rho}'}\kappa\Big(\frac{x'}{|x'|}\Big)dx'\right)^{-1}\int_{B_{\rho}'}\kappa\Big(\frac{x'}{|x'|}\Big)f(x')dx'.
\end{align*}

In order to prove Theorem \ref{thm001}, we need the following two propositions.
\begin{prop}\label{prop001}
For $d\geq 3$, $1+\sigma>0$, $1+\sigma\neq\alpha(\lambda_{1})$, let $\bar{v}\in H^{1}(B_{R_{0}}',|x'|^{m}dx')$ be a solution of
\begin{align}\label{KTMZ001}
\mathrm{div}(\delta\nabla\bar{v})=\mathrm{div}F,\quad\mathrm{in}\;B'_{R_{0}},
\end{align}
with $\varepsilon=0$ and $\|F\|_{\varepsilon,\sigma,0,B_{R_{0}}'}<\infty$. Then for any $\rho\in(0,R_{0})$,
\begin{align*}
\left(\fint_{\partial B_{\rho}'}|\bar{v}-\bar{v}(0')|^{2}\right)^{1/2}\leq C\|F\|_{\varepsilon,\sigma,0,B_{R_{0}}'}\rho^{\tilde{\alpha}(\lambda_{1})},
\end{align*}
where
\begin{align}\label{alpha001}
\tilde{\alpha}(\lambda_{1})=&\min\{\alpha(\lambda_{1}),1+\sigma\},
\end{align}
with $\alpha(\lambda_{1})$ given by \eqref{degree}.

\end{prop}

Introduce some constants as follows:
\begin{align}
\theta_{1}=&\Bigg[\kappa_{0}^{2}\bigg(\sum_{i\in\mathcal{A}}\kappa_{i}^{2}\bigg)^{m/2}+\sum_{j\in\mathcal{B}}\kappa_{j}^{2}\Bigg]^{1/2},\label{CONSTANT001}\\
\theta_{2}=&\Bigg[\kappa_{0}^{4}\bigg(\sum_{i\in\mathcal{A}}\kappa_{i}^{4}\bigg)\bigg(\sum_{i\in\mathcal{A}}\kappa_{i}^{2}\bigg)^{m-2}+\sum_{j\in\mathcal{B}}\kappa_{j}^{4}\Bigg]^{1/4},\label{CONSTANT002}\\
\theta_{3}=&\left(2^{-\frac{m-2}{2}}\min\big\{2^{-\frac{(m-2)(\mathfrak{b}-1)}{2}},1\big\}\min\Big\{\kappa_{0}\min_{i\in\mathcal{A}}\kappa_{i}^{m/2},\min_{j\in\mathcal{B}}\kappa_{j}\Big\}\right)^{1/m-1},\label{CONSTANT003}
\end{align}
where $\mathfrak{b}:=\mathrm{card}(\mathcal{B})$ denotes the number of elements in set $\mathcal{B}$. Define
\begin{align}\label{cvalue001}
c_{0}:=\min\bigg\{\frac{1}{4\big(1+\theta_{1}\big)^{1/m}},\frac{1}{2^{m}m\max\{\theta_{2},1\}\max\{\theta_{3},1\}}\bigg\}.
\end{align}
For $0<\rho<R_{0}$, define
\begin{align}\label{MZA001}
B'_{(1\pm\bar{c}_{0})\rho}:=B'_{(1+\bar{c}_{0})\rho}\setminus B'_{(1-\bar{c}_{0})\rho},\quad
\text{with }\bar{c}_{0}:=2c_{0}(1+\theta_{1})^{1/m}.
\end{align}

Remark that these parameters $\theta_{i}$, $i=1,2,3$, $c_{0}$ and $\bar{c}_{0}$ are introduced to achieve the following two goals. The first goal is to give a precise characterization in terms of the equivalence of the height of small thin gap in \eqref{QWN001}. The second one is to choose a suitable small neighbourhood centered at every point of the considered thin gap for the purpose of using the ``flipping argument" developed in \cite{BLY2010} to derive the pointwise upper bounds on the gradient in the following. See pages 15--21 below for further explanations and more details.

\begin{prop}\label{prop003}
For $d\geq 3$, $1+\sigma>0$, $1+\sigma\neq\alpha(\lambda_{1})$, let $\bar{v}\in H^{1}(B_{R_{0}}')$ be a solution of
\begin{align*}
\mathrm{div}(\delta\nabla\bar{v})=\mathrm{div}F,\quad\mathrm{in}\;B'_{R_{0}},
\end{align*}
with $\varepsilon>0$, $\|F\|_{\varepsilon,\sigma,0,B_{R_{0}}'}<\infty$ and $\|\nabla\bar{v}\|_{\varepsilon;-\tau,1,B_{R_{0}}'}<\infty$ for $\tau\leq1$.  Then for $0<\rho<(1-\bar{c}_{0})^{2}R\leq (1-\bar{c}_{0})^{2}R_{0}$,
\begin{align*}
&\left(\fint_{B'_{(1\pm\bar{c}_{0})\rho}}\Big|\bar{v}(x')-(\bar{v})^{\kappa}_{B'_{(1\pm\bar{c}_{0})\rho}}\Big|^{2}\right)^{1/2}\notag\\
&\leq C\Big(\frac{\rho}{R}\Big)^{\alpha(\lambda_{1})}\left(\fint_{B'_{(1\pm\bar{c}_{0})R}}\Big|\bar{v}(x')-(\bar{v})^{\kappa}_{B'_{(1\pm\bar{c}_{0})R}}\Big|^{2}\right)^{1/2}\notag\\
&+C\Big(\frac{R}{\rho}\Big)^{\frac{d+m-2}{2}}\left[R^{1+\sigma}\left(\frac{\sqrt{\varepsilon}}{R^{m/2}}+1\right)\|F\|_{\varepsilon,\sigma,0,B_{R_{0}}'}+\Big(\frac{\varepsilon}{R^{m}}\Big)^{\beta(\lambda_{1})}R^{1-\tau}\|\nabla\bar{v}\|_{\varepsilon;-\tau,1,B_{R_{0}}'}\right],
\end{align*}
where $\alpha(\lambda_{1})$ is defined by \eqref{degree}, and
\begin{align}\label{beta001}
\beta(\lambda_{1})=&
\begin{cases}
\frac{2\alpha(\lambda_{1})+d+m-3}{2m},&m>d+2\alpha(\lambda_{1})-3,\\
\text{any }\alpha<1,&m=d+2\alpha(\lambda_{1})-3,\\
1,&m<d+2\alpha(\lambda_{1})-3.
\end{cases}
\end{align}

\end{prop}

To prove Propositions \ref{prop001} and \ref{prop003}, we start by decomposing the solution $\bar{v}$ of \eqref{KTMZ001} as follows:
\begin{align}\label{ADAD001}
\bar{v}:=\bar{v}_{1}+\bar{v}_{2},\quad\mathrm{in}\;B_{R}',\;0<R\leq R_{0},
\end{align}
where $\bar{v}_{i},i=1,2,$ respectively, solve
\begin{align}\label{de001}
\begin{cases}
\mathrm{div}(\delta\nabla\bar{v}_{1})=0,& \mathrm{in}\;B_{R}',\\
\bar{v}_{1}=\bar{v},&\mathrm{on}\;\partial B_{R}',
\end{cases}
\end{align}
and
\begin{align}\label{de002}
\begin{cases}
\mathrm{div}(\delta\nabla\bar{v}_{2})=\mathrm{div}F,&\mathrm{in}\;B_{R}',\\
\bar{v}_{2}=0,&\mathrm{on}\;\partial B_{R}'.
\end{cases}
\end{align}

With regard to $\bar{v}_{1}$, we obtain
\begin{lemma}\label{lemma001}
For $d\geq 3$, let $\bar{v}_{1}\in H^{1}(B_{R}',|x'|^{m}dx')$ be a solution of \eqref{de001} with $\varepsilon=0$. Then $\bar{v}_{1}\in C^{\beta}(B_{R}')$ for some $\beta=\beta(d,m,\|\ln\kappa\|_{L^{\infty}(\mathbb{S}^{d-2})})$. Furthermore, for any $0<\rho<R$, we have
\begin{align*}
\bar{v}_{1}(0')=(\bar{v}_{1})^{\kappa}_{\partial B_{\rho}'},
\end{align*}
and
\begin{align*}
\left(\fint_{\partial B_{\rho}'}\kappa\Big(\frac{x'}{|x'|}\Big)|\bar{v}_{1}-\bar{v}_{1}(0')|^{2}\right)^{\frac{1}{2}}\leq\left(\frac{\rho}{R}\right)^{\alpha(\lambda_{1})}\left(\fint_{\partial B_{R}'}\kappa\Big(\frac{x'}{|x'|}\Big)|\bar{v}_{1}-\bar{v}_{1}(0')|^{2}\right)^{\frac{1}{2}},
\end{align*}
where $\alpha(\lambda_{1})$ is given in \eqref{degree}.

\end{lemma}
\begin{proof}
To begin with, in view of Theorem 2.3.12 and Section 3 (see pp. 106) in \cite{FKS1982}, we derive that $\bar{v}_{1}\in C^{\beta}(B_{R}')$ for some $\beta=\beta(d,m,\|\ln\kappa\|_{L^{\infty}(\mathbb{S}^{d-2})})$. Without loss of generality, set $R=1$. Let $x'=(r,\xi)\in(0,1)\times\mathbb{S}^{d-2}$. Picking a test function $\phi(r)\psi(\xi)$ with $\phi(r)\in C^{\infty}_{c}((0,1))$ and $\psi(\xi)\in C^{\infty}(\mathbb{S}^{d-2})$, it follows from integration by parts that
\begin{align*}
&\int_{B_{1}'}\kappa(\xi)r^{m}\nabla\bar{v}_{1}\nabla(\phi\psi)\notag\\
&=\int^{1}_{0}\int_{\mathbb{S}^{d-2}}\kappa r^{m+d-2}\partial_{r}\bar{v}_{1}\phi'\psi+\kappa r^{m+d-4}\nabla_{\mathbb{S}^{d-2}}\bar{v}_{1}\nabla_{\mathbb{S}^{d-2}}\psi\phi\,d\xi dr\\
&=-\int_{B_{1}'}[\kappa r^{m}\partial_{rr}\bar{v}_{1}+\kappa(m+d-2)r^{m-1}\partial_{r}\bar{v}_{1}+r^{m-2}\mathrm{div}_{\mathbb{S}^{d-2}}(\kappa\nabla_{\mathbb{S}^{d-2}}\bar{v}_{1})]\phi\psi.
\end{align*}
Hence $\bar{v}_{1}$ satisfies
\begin{align}\label{ZKCT001}
\partial_{rr}\bar{v}_{1}+\frac{m+d-2}{r}\partial_{r}\bar{v}_{1}+\frac{1}{\kappa(\xi)r^{2}}\mathrm{div}_{\mathbb{S}^{d-2}}(\kappa(\xi)\nabla_{\mathbb{S}^{d-2}}\bar{v}_{1})=0,\quad\mathrm{in}\;B_{1}'\setminus\{0'\}.
\end{align}
Set $\lambda_{0}=0$ and $\{\lambda_{i}\}_{i=1}^{\infty}$ denotes all the positive eigenvalues of \eqref{eigen001}, satisfying that $\lambda_{i}<\lambda_{i+1}$, $i\in\{0\}\cup\mathbb{N}$. Pick a positive constant $Y_{0}$ such that $\langle Y_{0},Y_{0}\rangle_{\mathbb{S}^{d-2}}=1$. Denote by $Y_{k,l}$ the corresponding eigenfunction of $\lambda_{k}$, satisfying
\begin{align*}
-\mathrm{div}_{\mathbb{S}^{d-2}}(\kappa(\xi)\nabla_{\mathbb{S}^{d-2}}Y_{k,l})=\lambda_{k}\kappa(\xi)Y_{k,l},\quad\xi\in\mathbb{S}^{d-2}.
\end{align*}
Moreover, $\{Y_{0}\}\cup\{Y_{k,l}\}_{k,l}$ comprises an orthonormal basis of $L^{2}(\mathbb{S}^{d-2})$ under the inner product \eqref{inner001}.

Decompose $\bar{v}_{1}$ as follows:
\begin{align}\label{FK001}
\bar{v}_{1}(r,\xi)=V_{0}(r)Y_{0}+\sum^{\infty}_{k=1}\sum^{N(k)}_{l=1}V_{k,l}(r)Y_{k,l}(\xi),\quad (r,\xi)\in(0,1)\times\mathbb{S}^{d-2},
\end{align}
where $V_{0}(r)$, $V_{k,l}(r)\in C^{2}(0,1)$ are, respectively, determined by
\begin{align*}
V_{0}(r)=\fint_{\mathbb{S}^{d-2}}\kappa(\xi)\bar{v}_{1}(r,\xi)Y_{0}d\xi,\quad V_{k,l}(r)=\fint_{\mathbb{S}^{d-2}}\kappa(\xi)\bar{v}_{1}(r,\xi)Y_{k,l}(\xi)d\xi.
\end{align*}
By using the test functions $\kappa(\xi)Y_{0}$ and $\kappa(\xi)Y_{k,l}(\xi)$ for equation \eqref{ZKCT001} on $\mathbb{S}^{d-2}$, we deduce that
\begin{align*}
V_{0}''+\frac{m+d-2}{r}V_{0}'=0,\quad V_{k,l}''+\frac{m+d-2}{r}V_{k,l}'-\frac{\lambda_{k}}{r^{2}}V_{k,l}=0,\quad0<r<1.
\end{align*}
A direct calculation gives that $V_{0}=c_{1}+c_{2}r^{3-m-d}$ and $V_{k,l}=c_{3}r^{\alpha(\lambda_{k})_{-}}+c_{4}r^{\alpha(\lambda_{k})_{+}}$ for some constants $c_{i}$, $i=1,2,3,4,$ where
\begin{align*}
\alpha(\lambda_{k})_{\pm}:=\frac{-(m+d-3)\pm\sqrt{(m+d-3)^{2}+4\lambda_{k}}}{2}.
\end{align*}
Claim that $c_{2}=0$. Otherwise, if $c_{2}\neq0$, then for any $\varrho>0$,
\begin{align*}
\int_{B_{1}'\setminus B_{\varrho}'}\kappa(\xi)\bar{v}_{1}^{2}r^{m}dx'\geq&\frac{1}{C}\int_{B_{1}'\setminus B_{\varrho}'}\kappa(\xi)V_{0}(r)^{2}r^{m}dx'\notag\\
\geq&\frac{1}{C}\int^{1}_{\varrho}|c_{1}+c_{2}r^{3-m-d}|^{2}r^{m+d-2}dr\notag\\
\geq&\frac{1}{C}\begin{cases}
|\ln\varrho|,&m=2,\,d=3,\\
\varrho^{5-m-d},&\text{otherwise},
\end{cases}\rightarrow\infty,\quad\text{as $\varrho\rightarrow0$}.
\end{align*}
This contradicts the assumed condition that $\bar{v}_{1}\in H^{1}(B_{R}',|x'|^{m}dx')$. Therefore, $c_{2}=0$. By the same argument, we also obtain that $c_{3}=0$. Consequently, we have
\begin{align}\label{NCDZ01}
V_{0}(\rho)\equiv V_{0}(1),\quad V_{k,l}(\rho)=\rho^{\alpha(\lambda_{k})_{+}}V_{k,l}(1),\quad\rho\in(0,1).
\end{align}
This, together with \eqref{FK001}, shows that
\begin{align*}
\fint_{\partial B_{\rho}'}\kappa\Big(\frac{x'}{|x'|}\Big)|\bar{v}_{1}-\bar{v}_{1}(0')|^{2}=&\sum^{\infty}_{k=1}\sum^{N(k)}_{l=1}|V_{k,l}(\rho)|^{2}\notag\\
\leq&\rho^{2\alpha(\lambda_{1})_{+}}\sum^{\infty}_{k=1}\sum^{N(k)}_{l=1}|V_{k,l}(1)|^{2}\notag\\
=&\rho^{2\alpha(\lambda_{1})_{+}}\fint_{\partial B_{1}'}\kappa\Big(\frac{x'}{|x'|}\Big)|\bar{v}_{1}-\bar{v}_{1}(0')|^{2},
\end{align*}
where we used the fact that $\bar{v}(0')=V_{0}(0)Y_{0}$. Moreover, we have from \eqref{NCDZ01} that $\bar{v}(0')=V_{0}(\rho)Y_{0}=(\bar{v}_{1})^{\kappa}_{\partial B_{\rho}'}$ for any $0<\rho<1$. The proof is complete.

\end{proof}

For $\bar{v}_{2}$, we establish its $L^{\infty}$ estimate by applying the Moser iteration in the following.
\begin{lemma}\label{lemma002}
For $d\geq 3$, $1+\sigma>0$, let $\bar{v}_{2}\in H^{1}_{0}(B_{R}',|x'|^{m}dx')$ be a solution of \eqref{de002} with $\varepsilon=0$ and $R=1$. Assume that $F\in L^{\infty}(B_{1}')$ and $\|F\|_{\varepsilon,\sigma,0,B_{1}'}<\infty$.  Then,
\begin{align*}
\|\bar{v}_{2}\|_{L^{\infty}(B_{1}')}\leq C\|F\|_{\varepsilon,\sigma,0,B_{1}'},
\end{align*}
where $C$ is a positive constant depending only on $d,m,\sigma$ and $\kappa_{i}$, $i=0,1,...,d-1$, but not on $\varepsilon$.
\end{lemma}
\begin{proof}
Without loss of generality, suppose that $\|F\|_{\varepsilon,\gamma,\sigma,B_{1}'}=1$.  Multiplying \eqref{de002} by $-|\bar{v}_{2}|^{p-2}\bar{v}_{2}$ with $p\geq2$ and integrating by parts, we derive
\begin{align*}
&(p-1)\int_{B_{1}'}\kappa\Big(\frac{x'}{|x'|}\Big)|x'|^{m}|\nabla\bar{v}_{2}|^{2}|\bar{v}_{2}|^{p-2}= (p-1)\int_{B'_{1}}F\cdot\nabla\bar{v}_{2}|\bar{v}_{2}|^{p-2}.
\end{align*}
Recall the following discrete version of H\"{o}lder's inequality
\begin{align}\label{INEQ001}
\left(\sum^{N}_{i=1}a_{i}b_{i}\right)^{2}\leq\left(\sum^{N}_{i=1}a_{i}^{2}\right)\left(\sum^{N}_{i=1}b_{i}^{2}\right),\quad a_{i},b_{i}\in\mathbb{R},\;i=1,...,N,\;N\geq1,
\end{align}
and the elemental inequality
\begin{align}\label{INEQ009}
a^{p}+b^{p}\leq(a+b)^{p}\leq2^{p-1}(a^{p}+b^{p}),\quad a,b>0,\;p\geq1.
\end{align}
For simplicity, denote $\xi=\frac{x'}{|x'|}\in\mathbb{S}^{d-2}$. A combination of \eqref{INEQ001}--\eqref{INEQ009} shows that
\begin{align}\label{MAME001}
\kappa(\xi)\leq&\kappa_{0}\bigg(\sum_{i\in\mathcal{A}}\kappa_{i}^{2}\bigg)^{\frac{m}{4}}\bigg(\sum_{i\in\mathcal{A}}|\xi_{i}|^{4}\bigg)^{\frac{m}{4}}+\bigg(\sum_{j\in\mathcal{B}}\kappa_{j}^{2}\bigg)^{\frac{1}{2}}\bigg(\sum_{j\in\mathcal{B}}|\xi_{j}|^{2m}\bigg)^{\frac{1}{2}}\notag\\
\leq&\kappa_{0}\bigg(\sum_{i\in\mathcal{A}}\kappa_{i}^{2}\bigg)^{\frac{m}{4}}\bigg(\sum_{i\in\mathcal{A}}|\xi_{i}|^{2}\bigg)^{\frac{m}{2}}+\bigg(\sum_{j\in\mathcal{B}}\kappa_{j}^{2}\bigg)^{\frac{1}{2}}\bigg(\sum_{j\in\mathcal{B}}|\xi_{j}|^{2}\bigg)^{\frac{m}{2}}\notag\\
\leq&\Bigg[\kappa_{0}^{2}\bigg(\sum_{i\in\mathcal{A}}\kappa_{i}^{2}\bigg)^{\frac{m}{2}}+\sum_{j\in\mathcal{B}}\kappa_{j}^{2}\Bigg]^{\frac{1}{2}}\Bigg[\bigg(\sum_{i\in\mathcal{A}}|\xi_{i}|^{2}\bigg)^{m}+\bigg(\sum_{j\in\mathcal{B}}|\xi_{j}|^{2}\bigg)^{m}\Bigg]^{\frac{1}{2}}\notag\\
\leq&\Bigg[\kappa_{0}^{2}\bigg(\sum_{i\in\mathcal{A}}\kappa_{i}^{2}\bigg)^{\frac{m}{2}}+\sum_{j\in\mathcal{B}}\kappa_{j}^{2}\Bigg]^{\frac{1}{2}},
\end{align}
and
\begin{align}\label{MAME002}
\kappa(\xi)\geq&\kappa_{0}\min_{i\in\mathcal{A}}\kappa_{i}^{\frac{m}{2}}\bigg(\sum_{i\in\mathcal{A}}|\xi_{i}|^{2}\bigg)^{\frac{m}{2}}+\min_{j\in\mathcal{B}}\kappa_{j}\sum_{j\in\mathcal{B}}|\xi_{j}|^{m}\notag\\
\geq&\min\Big\{\kappa_{0}\min_{i\in\mathcal{A}}\kappa_{i}^{\frac{m}{2}},\min_{j\in\mathcal{B}}\kappa_{j}\Big\}\Bigg[\bigg(\sum_{i\in\mathcal{A}}|\xi_{i}|^{2}\bigg)^{\frac{m}{2}}+\sum_{j\in\mathcal{B}}|\xi_{j}|^{m}\Bigg]\notag\\
\geq&2^{-\frac{m-2}{2}}\min\big\{2^{-\frac{(m-2)(\mathfrak{b}-1)}{2}},1\big\}\min\Big\{\kappa_{0}\min_{i\in\mathcal{A}}\kappa_{i}^{\frac{m}{2}},\min_{j\in\mathcal{B}}\kappa_{j}\Big\},
\end{align}
where in the last line we used the fact that
\begin{align*}
&\bigg(\sum_{i\in\mathcal{A}}|\xi_{i}|^{2}\bigg)^{\frac{m}{2}}+\sum_{j\in\mathcal{B}}|\xi_{j}|^{m}\notag\\
&\geq\min\big\{2^{-\frac{(m-2)(\mathfrak{b}-1)}{2}},1\big\}\Bigg[\bigg(\sum_{i\in\mathcal{A}}|\xi_{i}|^{2}\bigg)^{\frac{m}{2}}+\bigg(\sum_{j\in\mathcal{B}}|\xi_{j}|^{2}\bigg)^{\frac{m}{2}}\Bigg]\notag\\
&\geq2^{-\frac{m-2}{2}}\min\big\{2^{-\frac{(m-2)(\mathfrak{b}-1)}{2}},1\big\},
\end{align*}
with $\mathfrak{b}:=\mathrm{card}(\mathcal{B})$ representing the number of elements in set $\mathcal{B}$. Combining \eqref{MAME001} and \eqref{MAME002}, we obtain
\begin{align}\label{DZM001}
\varrho:=\theta_{3}^{-m/(m-1)}\leq\kappa\Big(\frac{x'}{|x'|}\Big)\leq\theta_{1},
\end{align}
where $\theta_{1}$ and $\theta_{3}$ are defined by \eqref{CONSTANT001} and \eqref{CONSTANT003}, respectively. Note that $|F|\leq|x'|^{\sigma+m}$ in $B_{1}'$. In light of $1+\sigma>0$, we then deduce from Young's inequality and H\"{o}lder's inequality that
\begin{align*}
&\left|\int_{B'_{1}}F\cdot\nabla\bar{v}_{2}|\bar{v}_{2}|^{p-2}\right|\notag\\
&\leq\frac{\varrho}{2}\int_{B_{1}'}|x'|^{m}|\nabla\bar{v}_{2}|^{2}|\bar{v}_{2}|^{2}+C\int_{B_{1}'}|x'|^{2\sigma+m}|\bar{v}_{2}|^{p-2}\notag\\
&\leq\frac{\varrho}{2}\int_{B_{1}'}|x'|^{m}|\nabla\bar{v}_{2}|^{2}|\bar{v}_{2}|^{2}\notag\\
&\quad+C\||\bar{v}_{2}|^{p-2}\|_{L^{\frac{d+m-1+2\mu}{d+m-3+2\mu}}(B_{1}',|x'|^{m}dx')}\left(\int_{B_{1}'}|x'|^{\sigma(d+m-1+2\mu)+m}\right)^{\frac{2}{d+m-1+2\mu}}\notag\\
&\leq\frac{\varrho}{2}\int_{B_{1}'}|x'|^{m}|\nabla\bar{v}_{2}|^{2}|\bar{v}_{2}|^{2}+C\||\bar{v}_{2}|^{p-2}\|_{L^{\frac{d+m-1+2\mu}{d+m-3+2\mu}}(B_{1}',|x'|^{m}dx')},
\end{align*}
where $\mu$ is chosen sufficiently small to ensure that
\begin{align*}
\left(\int_{B_{1}'}|x'|^{\sigma(d+m-1+2\mu)+m}\right)^{\frac{2}{d+m-1+2\mu}}<\infty.
\end{align*}
Consequently, combining these above facts, we deduce
\begin{align}\label{ZFMZ00A}
\frac{4}{p^{2}}\int_{B_{1}'}|x'|^{m}\big|\nabla|\bar{v}_{2}|^{\frac{p}{2}}\big|^{2}=&\int_{B_{1}'}|x'|^{m}|\nabla\bar{v}_{2}|^{2}|\bar{v}_{2}|^{p-2}\notag\\
\leq&C\||\bar{v}_{2}|^{p-2}\|_{L^{\frac{d+m-1+2\mu}{d+m-3+2\mu}}(B_{1}',|x'|^{m}dx')}.
\end{align}
We will utilize the Caffarelli-Kohn-Nirenberg inequality in \cite{CKN1984} having the following form:
\begin{align}\label{CKN00A}
\|u\|_{L^{\frac{2(d+m-1)}{d+m-3}}(B_{1}',|x'|^{m}dx')}\leq C\|\nabla u\|_{L^{2}(B_{1}',|x'|^{m}dx')},\quad\forall u\in H_{0}^{1}(B_{1}',|x'|^{m}dx').
\end{align}
Picking $p=2$ in \eqref{ZFMZ00A} and applying \eqref{CKN00A} with $u=|\bar{v}_{2}|$, we have from H\"{o}lder's inequality that
\begin{align}\label{FDTZ001}
\|\bar{v}_{2}\|_{L^{\frac{2(d+m-1+2\mu)}{d+m-3+2\mu}}(B_{1}',|x'|^{m}dx')}\leq C.
\end{align}
For $p\geq2$, applying \eqref{CKN00A} with $u=|\bar{v}_{2}|^{\frac{p}{2}}$ and using H\"{o}lder's inequality again, we obtain from \eqref{ZFMZ00A} that
\begin{align*}
\|\bar{v}_{2}\|^{p}_{L^{\frac{(d+m-1)p}{d+m-3}}(B_{1}',|x'|^{m}dx')}\leq&C\|\nabla|\bar{v}_{2}|^{\frac{p}{2}}\|^{2}_{L^{2}(B_{1}',|x'|^{m}dx')}\notag\\
\leq& Cp^{2}\|\bar{v}_{2}\|^{p-2}_{L^{\frac{(d+m-1+2\mu)(p-2)}{d+m-3+2\mu}}(B_{1}',|x'|^{m}dx')}\notag\\
\leq& Cp^{2}\|\bar{v}_{2}\|^{p-2}_{L^{\frac{(d+m-1+2\mu)p}{d+m-3+2\mu}}(B_{1}',|x'|^{m}dx')},
\end{align*}
which, in combination with Young's inequality, reads that
\begin{align*}
\|\bar{v}_{2}\|_{L^{\frac{(d+m-1)p}{d+m-3}}(B_{1}',|x'|^{m}dx')}\leq&(Cp^{2})^{1/p}\left(\|\bar{v}_{2}\|_{L^{\frac{(d+m-1+2\mu)p}{d+m-3+2\mu}}(B_{1}',|x'|^{m}dx')}+\frac{2}{p}\right).
\end{align*}
Denote
\begin{align*}
p_{k}=\frac{2(d+m-1+2\mu)}{d+m-3+2\mu}\left(\frac{d+m-1}{d+m-3}\cdot\frac{d+m-3+2\mu}{d+m-1+2\mu}\right)^{k},\quad k\geq0.
\end{align*}
After $k$ iterations, it follows from \eqref{FDTZ001} that
\begin{align*}
\|\bar{v}_{2}\|_{L^{p_{k}}(B_{1}',|x'|^{m}dx')}\leq&\prod^{k-1}_{i=0}(Cp_{i}^{2})^{1/p_{i}}\|\bar{v}_{2}\|_{L^{p_{0}}(B_{1}',|x'|^{m}dx')}\notag\\
&+\sum^{k-1}_{i=0}\prod^{k-1-i}_{j=0}(Cp^{2}_{k-1-j})^{1/p_{k-1-j}}\frac{2}{p_{i}}\notag\\
\leq&C\|\bar{v}_{2}\|_{L^{\frac{2(d+m-1+2\mu)}{d+m-3+2\mu}}(B_{1}',|x'|^{m}dx')}+C\sum^{k-1}_{i=0}\frac{1}{p_{i}}\leq C,
\end{align*}
where $C=C(d,m,\sigma)$. Sending $k\rightarrow\infty$, we complete the proof of Lemma \ref{lemma002}.

\end{proof}
In order to prove Proposition \ref{prop003}, we also need the following lemma.
\begin{lemma}\label{lemmaC01}
Let $\bar{w}\in H^{1}_{0}(B_{1}',|x'|^{m+\beta}dx')$ with $d\geq3,\,m\geq2$ and $\beta<1$. Then we have
\begin{align*}
\sup\limits_{0<r<1}r^{d+m-2}\fint_{\partial B_{r}'}|\bar{w}|^{2}\leq C\int_{B_{1}'}|x'|^{m+\beta}|\nabla\bar{w}|^{2},
\end{align*}
where $C=C(d,m,\beta).$
\end{lemma}

\begin{proof}
By denseness, it suffices to consider $\bar{u}\in C^{1}(B_{1}')$. Utilizing H\"{o}lder's inequality and the Fubini theorem, we obtain
\begin{align*}
r^{m}\int_{\partial B_{r}'}|\bar{w}|^{2}dx'=&\int_{\mathbb{S}^{d-2}}r^{d+m-2}|\bar{w}(r,\xi)|^{2}d\xi\notag\\
=&\int_{\mathbb{S}^{d-2}}r^{d+m-2}\left(\int^{1}_{r}\partial_{s}\bar{w}(s,\xi)ds\right)^{2}d\xi\notag\\
\leq&\int_{\mathbb{S}^{d-2}}r^{d+m-2}\left(\int^{1}_{r}|\partial_{s}\bar{w}(s,\xi)|^{2}s^{\beta}ds\right)\left(\int^{1}_{r}s^{-\beta}ds\right)d\xi\notag\\
\leq&C\int_{0}^{1}\int_{\mathbb{S}^{d-2}}s^{d+m-2+\beta}|\partial_{s}\bar{w}(s,\xi)|^{2}dsd\xi\leq C\int_{B_{1}'}|x'|^{m+\beta}|\nabla\bar{w}|^{2}dx'.
\end{align*}
The proof is complete.
\end{proof}

We are now ready to use Lemmas \ref{lemma001}, \ref{lemma002} and \ref{lemmaC01} to give the proofs of Propositions \ref{prop001} and \ref{prop003}.
\begin{proof}[Proof of Propositions \ref{prop001} and \ref{prop003}]
We divide into two parts to complete the proofs.

{\bf Part 1.} Consider the case when $\varepsilon=0$. Let $\bar{v}(0')=0$ and $\|F\|_{0,\sigma,0,B_{R_{0}}'}=1$ without loss of generality. For $0<\rho\leq R\leq R_{0}$, write
\begin{align*}
\omega(\rho):=\bigg(\fint_{\partial B_{\rho}'}\kappa\Big(\frac{x'}{|x'|}\Big)|\bar{v}|^{2}\bigg)^{\frac{1}{2}}.
\end{align*}
Denote $\tilde{v}_{2}(x'):=\bar{v}_{2}(Rx')$ and $\tilde{F}(y'):=R^{-(m-1)}F(Ry')$. From \eqref{de002}, we obtain that $\tilde{v}_{2}$ solves
\begin{align*}
\mathrm{div}\Big[\kappa\Big(\frac{x'}{|x'|}\Big)|x'|^{m}\nabla\tilde{v}_{2}\Big]=\mathrm{div}\tilde{F},\quad\mathrm{in}\;B_{1}',
\end{align*}
where $\|\tilde{F}\|_{0,\sigma,0,B_{1}'}=R^{1+\sigma}\|F\|_{0,\sigma,\gamma,B_{R}'}$. Applying Lemma \ref{lemma002} to $\tilde{v}_{2}$, we have
\begin{align*}
\|\bar{v}_{2}\|_{L^{\infty}(B_{R}')}\leq CR^{1+\sigma}.
\end{align*}
This, together with \eqref{ADAD001} and Lemma \ref{lemma001}, shows that
\begin{align}\label{GAZ001}
\omega(\rho)\leq&\bigg(\fint_{\partial B_{\rho}'}\kappa\Big(\frac{x'}{|x'|}\Big)|\bar{v}_{1}-\bar{v}_{1}(0')|^{2}\bigg)^{\frac{1}{2}}+\bigg(\fint_{\partial B_{\rho}'}\kappa\Big(\frac{x'}{|x'|}\Big)|\bar{v}_{2}-\bar{v}_{2}(0')|^{2}\bigg)^{\frac{1}{2}}\notag\\
\leq&\left(\frac{\rho}{R}\right)^{\alpha(\lambda_{1})}\bigg(\fint_{\partial B_{R}'}\kappa\Big(\frac{x'}{|x'|}\Big)|\bar{v}_{1}|^{2}\bigg)^{\frac{1}{2}}+\left(\frac{\rho}{R}\right)^{\alpha}|\bar{v}_{1}(0')|+2\|\bar{v}_{2}\|_{L^{\infty}(B_{R}')}\notag\\
\leq&\left(\frac{\rho}{R}\right)^{\alpha(\lambda_{1})}\omega(R)+CR^{1+\sigma},
\end{align}
where the facts that $\bar{v}=\bar{v}_{1}$ on $\partial B_{R}'$ and $|\bar{v}_{1}(0')|=|\bar{v}_{2}(0')|$ were utilized. Pick $\rho=2^{-i-1}R_{0}$ and $R=2^{-i}R_{0}$ in \eqref{GAZ001} with $i=0,...,k-1,$ $k$ is a positive integer. After $k$ iterations and in light of $1+\sigma\neq\alpha(\lambda_{1})$, we obtain
\begin{align*}
\omega(2^{-k}R_{0})\leq&2^{-k\alpha(\lambda_{1})}\omega(R_{0})+C\sum^{k}_{i=1}2^{-(k-i)\alpha(\lambda_{1})}(2^{1-i}R_{0})^{1+\sigma}\notag\\
\leq&2^{-k\alpha(\lambda_{1})}\omega(R_{0})+C2^{-k\alpha(\lambda_{1})}R_{0}^{1+\sigma}\frac{1-2^{k(\alpha(\lambda_{1})-1-\sigma)}}{1-2^{\alpha(\lambda_{1})-1-\sigma}}\notag\\
\leq&2^{-k\tilde{\alpha}(\lambda_{1})}\big(\omega(R_{0})+CR_{0}^{1+\sigma}\big),
\end{align*}
where $\tilde{\alpha}(\lambda_{1})$ is given by \eqref{alpha001}. Observe that for any $\rho\in(0,R_{0})$, there exists some integer $k$ such that $\rho\in(2^{-k-1}R_{0},2^{-k}R_{0}]$. Then we derive
\begin{align*}
\omega(\rho)\leq C\rho^{\tilde{\alpha}(\lambda_{1})},\quad\mathrm{for}\;\mathrm{any}\;\rho\in(0,R_{0}).
\end{align*}
which, together with \eqref{DZM001}, yields that Proposition \ref{prop001} holds.

{\bf Part 2.} Consider the case when $\varepsilon>0$. Without loss of generality, assume that $\bar{v}(0')=0$. For simplicity, denote $\alpha:=\alpha(\lambda_{1})$ and $\beta:=\beta(\lambda_{1})$, which are, respectively, given by \eqref{degree} and \eqref{beta001}. To begin with, from the mean value formula, we know
\begin{align}\label{IZ001}
|\bar{v}(x')|\leq|x'|^{1-\tau}\|\nabla\bar{v}\|_{\varepsilon,-\tau,1,B_{R_{0}}'},\quad\text{in }B_{R_{0}}'.
\end{align}
For any $0<R<R_{0}$, denote $\bar{v}_{2}:=\bar{v}-\bar{v}_{1}$, where $\bar{v}\in H^{1}(B_{R}')$ solves equation \eqref{KTMZ001} with $\varepsilon>0$ and $\bar{v}_{1}\in H^{1}(B_{R}',|x'|^{m}dx')$ satisfies equation \eqref{de001} with $\varepsilon=0$. From Lemma \ref{lemma001} and \eqref{IZ001}, we obtain
\begin{align*}
\left(\fint_{\partial B_{\rho}'}|\bar{v}_{1}-\bar{v}_{1}(0')|^{2}\right)^{\frac{1}{2}}\leq C\Big(\frac{\rho}{R}\Big)^{\alpha}R^{1-t}\|\nabla\bar{v}\|_{\varepsilon,-\tau,1,B_{R_{0}}'},\quad0<\rho<R.
\end{align*}
For $0<\rho<R$, applying the interior estimate in $B_{\rho}\setminus\overline{B}_{\rho/2}$, we derive
\begin{align}\label{PZWZA001}
|\nabla\bar{v}_{1}(x')|\leq C|x'|^{\alpha-1}R^{1-t-\alpha}\|\nabla\bar{v}\|_{\varepsilon,-\tau,1,B'_{R_{0}}},\quad\mathrm{in}\;B_{R/2}'\setminus\{0'\}.
\end{align}
Combining Lemma \ref{lemma001}, \eqref{IZ001} and the maximum principle, we have
\begin{align}\label{PZWZA002}
\|\bar{v}_{1}\|_{L^{\infty}(B_{R}')}=\sup\limits_{x'\in\partial B_{R}'\cup\{0'\}}|\bar{v}_{1}(x')|\leq CR^{1-\tau}\|\nabla\bar{v}\|_{\varepsilon,-\tau,1,B_{R_{0}}'}.
\end{align}
It then follows from the boundary estimate that
\begin{align}\label{PZWZA003}
|\nabla\bar{v}_{1}(x')|\leq CR^{-\tau}\|\nabla\bar{v}\|_{\varepsilon,-\tau,1,B_{R_{0}}'},\quad\mathrm{in}\;B'_{R}\setminus B'_{R/2}.
\end{align}
From \eqref{PZWZA001}--\eqref{PZWZA003}, we see that $\bar{v}_{1}\in H^{1}(B_{R}')$. Hence $\bar{v}_{2}\in H^{1}_{0}(B_{R}')$ verifies
\begin{align*}
\mathrm{div}\left[\Big(\varepsilon+\kappa\Big(\frac{x'}{|x'|}\Big)|x'|^{m}\Big)\nabla\bar{v}_{2}\right]=\mathrm{div}F-\varepsilon\Delta\bar{v}_{1},\quad\mathrm{for}\;x'\in B_{R}'.
\end{align*}
Denote $\tilde{v}_{i}(y')=\bar{v}_{i}(Ry')$, $i=1,2,$ $\tilde{F}(y')=R^{-(m-1)}F(Ry')$ and $\tilde{\varepsilon}=\varepsilon R^{-m}$. Then we have
\begin{align}\label{EAZG001}
\|\tilde{F}\|_{\tilde{\varepsilon},\sigma,0,B_{1}'}=R^{1+\sigma}\|F\|_{\varepsilon,\sigma,0,B_{R}'},\quad\|\nabla\tilde{v}_{1}\|_{\tilde{\varepsilon},\alpha-1,1,B_{1}'}=R^{\alpha}\|\nabla\bar{v}_{1}\|_{\varepsilon,\alpha-1,1,B_{R}'},
\end{align}
and
\begin{align}\label{ENZM001}
\mathrm{div}\left[\Big(\tilde{\varepsilon}+\kappa\Big(\frac{x'}{|x'|}\Big)|x'|^{m}\Big)\nabla\tilde{v}_{2}\right]=\mathrm{div}\tilde{F}-\tilde{\varepsilon}\Delta\tilde{v}_{1},\quad\mathrm{for}\;x'\in B_{1}'.
\end{align}
From \eqref{PZWZA001} and \eqref{PZWZA003}--\eqref{EAZG001}, we obtain
\begin{align}\label{FAZ001}
\|\nabla\tilde{v}_{1}\|_{\tilde{\varepsilon},\alpha-1,1,B_{1}'}\leq CR^{1-\tau}\|\nabla\bar{v}\|_{\varepsilon.-\tau,1,B_{R_{0}}'}.
\end{align}
Set $x'=(r,\xi)\in(0,1)\times\mathbb{S}^{d-2}$. Then multiplying \eqref{ENZM001} by $\tilde{v}_{2}$, we have from integration by parts that
\begin{align*}
\int_{B_{1}'}(\tilde{\varepsilon}+\kappa(\xi)r^{m})|\nabla\tilde{v}_{2}|^{2}=\int_{B_{1}'}\tilde{F}\cdot\nabla\tilde{v}_{2}-\tilde{\varepsilon}\int_{B_{1}'}\nabla\tilde{v}_{1}\cdot\nabla\tilde{v}_{2}.
\end{align*}
For simplicity, denote $\|\tilde{F}\|=\|\tilde{F}\|_{\varepsilon,\sigma,0,B_{1}'}$ and $\|\nabla\tilde{v}_{1}\|=\|\nabla\tilde{v}_{1}\|_{\tilde{\varepsilon},\alpha-1,1,B_{1}'}$. In view of \eqref{DZM001} and $2\sigma+n-2>-1$, it follows from Young's inequality and Lemma \ref{lemmaC01} that
\begin{align*}
\sup\limits_{0<r<1}r^{d+m-2}\fint_{\partial B_{r}'}|\tilde{v}_{2}|^{2}\leq&\int_{B_{1}'}(\tilde{\varepsilon}+r^{m})|\nabla\tilde{v}_{2}|^{2}\notag\\
\leq&C\|\tilde{F}\|^{2}\int_{B_{1}'}r^{2\sigma}(\tilde{\varepsilon}+r^{m})+C\|\nabla\tilde{v}_{1}\|^{2}\int_{B_{1}'}\frac{\tilde{\varepsilon}^{2}r^{2\alpha-2}}{\tilde{\varepsilon}+r^{m}}\notag\\
\leq&C\|\tilde{F}\|^{2}(\tilde{\varepsilon}+1)+C\|\nabla\tilde{v}_{1}\|^{2}\tilde{\varepsilon}^{2\beta}.
\end{align*}
This, together with \eqref{EAZG001} and \eqref{FAZ001}, gives that for $0<\rho<R$,
\begin{align}\label{TGZWQ01}
\fint_{\partial B_{\rho}'}|\tilde{v}_{2}|^{2}\leq& C\Big(\frac{R}{\rho}\Big)^{d+m-2}R^{2+2\sigma}\Big(\frac{\varepsilon}{R^{m}}+1\Big)\|F\|^{2}_{\varepsilon,\sigma,0,B_{R_{0}}'}\notag\\
&+C\Big(\frac{R}{\rho}\Big)^{d+m-2}\Big(\frac{\varepsilon}{R^{m}}\Big)^{2\beta}R^{2-2\tau}\|\nabla\bar{v}\|^{2}_{\varepsilon,-\tau,1,B_{R_{0}}'}.
\end{align}
Then combining Lemma \ref{lemma001} and \eqref{TGZWQ01}, it follows that for any $0<\rho<(1-\bar{c}_{0})^{2}R$,
\begin{align*}
&\fint_{\partial B_{\rho}'}\kappa\Big(\frac{x'}{|x'|}\Big)|\bar{v}(x')-\bar{v}_{1}(0')|^{2}\notag\\
&\leq2\fint_{\partial B_{\rho}'}\kappa\Big(\frac{x'}{|x'|}\Big)|\bar{v}_{1}(x')-\bar{v}_{1}(0')|^{2}+2\fint_{\partial B_{\rho}'}\kappa\Big(\frac{x'}{|x'|}\Big)|\bar{v}_{2}(x')|^{2}\notag\\
&\leq C\Big(\frac{\rho}{R}\Big)^{2\alpha}\fint_{\partial B_{R}'}\kappa\Big(\frac{x'}{|x'|}\Big)|\bar{v}(x')-(\bar{v})^{\kappa}_{\partial B_{R}'}|^{2}\notag\\
&\quad+C\Big(\frac{R}{\rho}\Big)^{d+m-2}\left[R^{2+2\sigma}\Big(\frac{\varepsilon}{R^{m}}+1\Big)\|F\|^{2}_{\varepsilon,\sigma,0,B_{R_{0}}'}+\Big(\frac{\varepsilon}{R^{m}}\Big)^{2\beta}R^{2-2\tau}\|\nabla\bar{v}\|^{2}_{\varepsilon,-\tau,1,B_{R_{0}}'}\right].
\end{align*}
Multiplying the above by $\rho^{d-2}$ and integrating from $(1-\bar{c_{0}})\rho$ to $(1+\bar{c}_{0})\rho$, $\bar{c}_{0}$ is given in \eqref{MZA001}, we deduce that for any $0<\rho<(1-\bar{c}_{0})r\leq(1-\bar{c}_{0}^{2})R$,
\begin{align*}
&\fint_{B'_{(1\pm\bar{c}_{0})\rho}}\kappa\Big(\frac{x'}{|x'|}\Big)\Big|\bar{v}(x')-(\bar{v})^{\kappa}_{B'_{(1\pm\bar{c}_{0})\rho}}\Big|^{2}\leq\fint_{B'_{(1\pm\bar{c}_{0})\rho}}\kappa\Big(\frac{x'}{|x'|}\Big)|\bar{v}(x')-\bar{v}_{1}(0)|^{2}\notag\\
&\leq C\Big(\frac{\rho}{r}\Big)^{2\alpha}\fint_{\partial B_{r}'}\kappa\Big(\frac{x'}{|x'|}\Big)|\bar{v}(x')-(\bar{v})^{\kappa}_{\partial B_{r}'}|^{2}\notag\\
&\quad+C\Big(\frac{r}{\rho}\Big)^{d+m-2}\left[r^{2+2\sigma}\Big(\frac{\varepsilon}{r^{m}}+1\Big)\|F\|^{2}_{\varepsilon,\sigma,0,B_{r}'}+\Big(\frac{\varepsilon}{r^{m}}\Big)^{2\beta}r^{2-2\tau}\|\nabla\bar{v}\|^{2}_{\varepsilon,-\tau,1,B_{R_{0}}'}\right]\notag\\
&\leq C\Big(\frac{\rho}{r}\Big)^{2\alpha}\fint_{\partial B_{r}'}\kappa\Big(\frac{x'}{|x'|}\Big)|\bar{v}(x')-(\bar{v})^{\kappa}_{B'_{(1\pm\bar{c}_{0})R}}|^{2}\notag\\
&\quad+C\Big(\frac{r}{\rho}\Big)^{d+m-2}\left[r^{2+2\sigma}\Big(\frac{\varepsilon}{r^{m}}+1\Big)\|F\|^{2}_{\varepsilon,\sigma,0,B_{R_{0}}'}+\Big(\frac{\varepsilon}{r^{m}}\Big)^{2\beta}r^{2-2\tau}\|\nabla\bar{v}\|^{2}_{\varepsilon,-\tau,1,B_{R_{0}}'}\right],
\end{align*}
where $B'_{(1\pm\bar{c}_{0})\rho}$ is defined by \eqref{MZA001}. Therefore, multiplying it by $r^{d-2}$ and integrating from $(1-\bar{c}_{0})R$ to $(1+\bar{c}_{0})R$, we complete the proof of Proposition \ref{prop003}.

\end{proof}

For any given $x_{0}'=(x^{0}_{1},...,x^{0}_{d-1})\in B_{R_{0}/2}'$, denote
\begin{align*}
\delta(x_{0}'):=\varepsilon+\kappa_{0}\Big(\sum\limits_{i\in\mathcal{A}}\kappa_{i}|x^{0}_{i}|^{2}\Big)^{\frac{m}{2}}+\sum\limits_{j\in\mathcal{B}}\kappa_{j}|x^{0}_{j}|^{m}.
\end{align*}
For $s,t>0$ and $x'\in B'_{R_{0}}$, denote by $Q_{s,t}(x')$ the cylinder as follows:
\begin{align*}
Q_{s,t}(x'):=\{y=(y',y_{d})\in\mathbb{R}^{d}\,|\,|y'-x'|<s,\,|y_{d}|<t\}.
\end{align*}
For simplicity, let $Q_{s,t}:=Q_{s,t}(0')$ if $x'=0'$. For the convenience of presentation, in the following the domain notations such as $Q_{s,t}(x')$, $\Omega_{s}(x')$ and $B'_{t}$ are used in the sense that $Q_{s,t}(x')\setminus\{y'=0'\}$, $\Omega_{s}(x')\setminus\{y'=0'\}$ and $B_{t}'\setminus\{0'\}$ if $\varepsilon=0$ and $x'\neq0'$.

Similar to \eqref{INEQ001}--\eqref{DZM001}, a direct computation gives that if $\varepsilon=0$ and $x'_{0}\neq0'$, or if $\varepsilon>0$ and $|x'_{0}|\geq\varepsilon^{1/m}$,
\begin{align}\label{KDBAW001}
\theta_{3}^{-m/(m-1)}|x_{0}'|^{m}\leq\delta_{0}\leq (1+\theta_{1})|x_{0}'|^{m},
\end{align}
where $\theta_{1}$ and $\theta_{3}$ are, respectively, given by \eqref{CONSTANT001} and \eqref{CONSTANT003}. In view of the value of $c_{0}$ given in \eqref{cvalue001}, we know that $c_{0}\leq \frac{1}{4}\big(1+\theta_{1}\big)^{-1/m}$. This, together with \eqref{KDBAW001}, shows that
\begin{align}\label{DZNQ001}
|x'_{0}|-2c_{0}\delta_{0}^{1/m}\geq\frac{1}{2}|x'_{0}|,\;\,\text{if $\varepsilon=0$ and $x'_{0}\neq0'$, or if $\varepsilon>0$ and $|x'_{0}|\geq\varepsilon^{1/m}$}.
\end{align}
Observe that
\begin{align*}
\sum^{N}_{i=1}a_{i}^{p}\leq \bigg(\sum^{N}_{i=1}a_{i}\bigg)^{p},\quad a_{i}\geq0,\,i=1,...,N,\,N\geq1,\;p\geq1,
\end{align*}
which, together with \eqref{INEQ001}, reads that
\begin{align}\label{AWAZ001}
|\nabla_{x'}\delta(x')|^{2}=&m^{2}\kappa_{0}^{2}\sum_{i\in\mathcal{A}}\kappa_{i}^{2}x_{i}^{2}\bigg(\sum_{i\in\mathcal{A}}\kappa_{i}x_{i}^{2}\bigg)^{m-2}+m^{2}\sum_{j\in\mathcal{B}}\kappa_{j}^{2}|x_{j}|^{2(m-1)}\notag\\
\leq&m^{2}\kappa_{0}^{2}\bigg(\sum_{i\in\mathcal{A}}\kappa_{i}^{4}\bigg)^{\frac{1}{2}}\bigg(\sum_{i\in\mathcal{A}}\kappa_{i}^{2}\bigg)^{\frac{m-2}{2}}\bigg(\sum_{i\in\mathcal{A}}x_{i}^{4}\bigg)^{\frac{m-1}{2}}\notag\\
&+m^{2}\bigg(\sum_{j\in\mathcal{B}}\kappa_{j}^{4}\bigg)^{\frac{1}{2}}\bigg(\sum_{j\in\mathcal{B}}|x_{j}|^{4(m-1)}\bigg)^{\frac{1}{2}}\notag\\
\leq&m^{2}\theta_{2}^{2}\Bigg[\bigg(\sum_{i\in\mathcal{A}}x_{i}^{4}\bigg)^{m-1}+\sum_{j\in\mathcal{B}}|x_{j}|^{4(m-1)}\Bigg]^{\frac{1}{2}}\notag\\
\leq&m^{2}\theta_{2}^{2}\Bigg[\bigg(\sum_{i\in\mathcal{A}}x_{i}^{2}\bigg)^{2(m-1)}+\bigg(\sum_{j\in\mathcal{B}}x_{j}^{2}\bigg)^{2(m-1)}\Bigg]^{\frac{1}{2}}\notag\\
\leq&m^{2}\theta_{2}^{2}|x'|^{2(m-1)},
\end{align}
where $\theta_{2}$ is given by \eqref{CONSTANT002}. Recalling the value of $c_{0}$ given by \eqref{cvalue001}, it follows from \eqref{KDBAW001} and \eqref{AWAZ001} that for $0<s\leq2c_{0}\delta^{1/m}_{0}$, $x=(x',x_{d})\in\Omega_{s}(x_{0}')$,
\begin{align*}
|\delta(x')-\delta(x_{0}')|\leq& m\theta_{2}|x_{\theta}'|^{m-1}|x'-x_{0}'|\notag\\
\leq& m2^{m-2}\theta_{2}s(s^{m-1}+|x_{0}'|^{m-1})\notag\\
\leq&m2^{m-1}\theta_{2}\big((2c_{0})^{m-1}+\theta_{3}\big)c_{0}\leq\frac{\delta(x_{0}')}{2},
\end{align*}
where $x_{\theta}'$ is some point between $x'$ and $x_{0}'$. This leads to that
\begin{align}\label{QWN001}
\frac{1}{2}\delta(x_{0}')\leq\delta(x')\leq\frac{3}{2}\delta(x_{0}'),\quad\mathrm{in}\;\Omega_{s}(x_{0}').
\end{align}
From \eqref{QWN001}, we give a precise description for the equivalence of the height for small narrow region $\Omega_{s}(x_{0}')$.

For $\varepsilon\geq0$, let
\begin{align}\label{change001}
\begin{cases}
y'=x',\\
y_{d}=2\delta_{0}\left(\frac{x_{d}-h_{2}(x')+\varepsilon/2}{\varepsilon+h_{1}(x')-h_{2}(x')}-\frac{1}{2}\right).
\end{cases}
\end{align}
Under the change of variables \eqref{change001}, $\Omega_{R_{0}}$ becomes $Q_{R_{0},\delta_{0}}$. It is worth pointing out that by making use of the rescaling argument for every height line segment $\delta_{0}$ of the thin gap $\Omega_{R_{0}/2}$ in \eqref{change001}, we will present the proof for Theorem \ref{thm001} in a different style (by contrast with the proofs of Theorems 1.1 and 1.3 in pages 15--22 of \cite{DLY2022}). Our style of proof requires the exact value of every aforementioned parameter, including $\theta_{i}$, $i=1,2,3$, $c_{0}$ and $\bar{c}_{0}$ defined by \eqref{CONSTANT001}--\eqref{MZA001}. In the following these parameters make us succeed to apply the ``flipping argument" created in \cite{BLY2010} to a small neighbourhood centered at every point of the considered narrow region (specially, if $\varepsilon=0$, pick $x'_{0}\neq0'$, while if $\varepsilon>0$, choose $|x'_{0}|\geq\varepsilon^{1/m}$). Based on the above analysis, our style may be more favorable to deepen the readers' understanding on the idea and schemes developed in \cite{BLY2010,LY202102,DLY2022}.

Set $v(y)=u(x)$. In view of \eqref{problem006}, we obtain that $v$ verifies
\begin{align}\label{ZKM001}
\begin{cases}
-\partial_{i}(b_{ij}(y)\partial_{j}v(y))=0,&\mathrm{in}\;Q_{R_{0},\delta_{0}},\\
b_{dj}(y)\partial_{j}v(y)=0,&\mathrm{on}\;\{y_{d}=\pm\delta_{0}\},
\end{cases}
\end{align}
with $\|v\|_{L^{\infty}(Q_{R_{0},\varepsilon})}\leq1$ and
\begin{gather}
\begin{align*}
(b_{ij}(y))=&\frac{2\delta_{0}(\partial_{x}y)(A_{ij})(\partial_{x}y)^{t}}{\det(\partial_{x}y)}=\frac{2\delta_{0}(\partial_{x}y)(\partial_{x}y)^{t}}{\det(\partial_{x}y)}+\frac{2\delta_{0}(\partial_{x}y)(A_{ij}-\delta_{ij})(\partial_{x}y)^{t}}{\det(\partial_{x}y)}\notag\\
=&
\begin{pmatrix}\delta&0&\cdots&0&b^{1d} \\ 0&\delta&\cdots&0&b^{2d}\\ \vdots&\vdots&\ddots&\vdots&\vdots\\0&0&\cdots&\delta&b^{d-1\,1}\\ b^{d1}&b^{d2}&\cdots&b^{d\,d-1}&\frac{4\varepsilon^{2}+\sum^{d-1}_{i=1}a_{id}^{2}}{\delta}
\end{pmatrix}+\begin{pmatrix} e^{1}&0&\cdots&0 \\ 0&e^{2}&\cdots&0 \\ \vdots&\vdots&\ddots&\vdots\\
0&0&\cdots&e^{d}
\end{pmatrix}\notag\\
&+\begin{pmatrix} c^{11}&c^{12}&\cdots&c^{1d} \\ c^{21}&c^{22}&\cdots&c^{2d} \\ \vdots&\vdots&\ddots&\vdots\\
c^{d1}&c^{d2}&\cdots&c^{dd}
\end{pmatrix},
\end{align*}
\end{gather}
where elements satisfy that for $i=1,...,d-1,$ using conditions ({\bf{H1}}) and ({\bf{H2}}),
\begin{align}\label{K001}
|b^{id}|=|b^{di}|=|-2\delta_{0}\partial_{i}h_{2}(y')-(y_{d}+\delta_{0})\partial_{i}(h_{1}-h_{2})(y')|\leq C\delta_{0}|y'|^{m-1},
\end{align}
and
\begin{align}\label{K002}
|e^{i}|=|O(|y'|^{m+\gamma})|\leq C|y'|^{m+\gamma},\quad |e^{d}|\leq C\delta_{0}^{2}|y'|^{\gamma}\delta^{-1},
\end{align}
and every $c^{ij}$ is the element of the matrix $\frac{2\delta_{0}(\partial_{x}y)(A_{ij}-\delta_{ij})(\partial_{x}y)^{t}}{\det(\partial_{x}y)}$, satisfying that for $i,j=1,...,d-1$,
\begin{align*}
|c^{ij}|\leq C\delta(|y'|^{\gamma}+\delta^{\gamma}),\quad |c^{id}|\leq C\delta_{0}(|y'|^{\gamma}+\delta^{\gamma}).
\end{align*}

In light of \eqref{DM001}, a straightforward computation leads to that
\begin{align}\label{AZ005}
|\nabla_{y'}v|\leq \delta^{-1/m},\quad |\partial_{d} v|\leq C\delta_{0}^{-1}\delta^{1-1/m},\quad\mathrm{in}\;Q_{R_{0},\delta_{0}}.
\end{align}
Denote
\begin{align}\label{KL01}
\bar{v}(y'):=\fint^{\delta_{0}}_{-\delta_{0}}v(y',y_{d})\,dy_{d}.
\end{align}
Thus $\bar{v}$ solves
\begin{align}\label{ZK003}
\mathrm{div}(\delta\nabla\bar{v})=\mathrm{div}F,\quad\mathrm{in}\;B'_{R_{0}},
\end{align}
where $F=(F_{1},...,F_{d-1})$, $F_{i}:=-\overline{b^{id}\partial_{d}v}-e^{i}\partial_{i}\bar{v}-\sum^{d}_{j=1}\overline{c^{ij}\partial_{j}v}$ for $i=1,...,d-1$, $\overline{b^{id}\partial_{d}v}$ and $\overline{c^{ij}\partial_{j}v}$ denotes, respectively, the averages of $b^{id}\partial_{d}v$ and $c^{ij}\partial_{j}v$ with respect to $y_{d}$ on $(-\delta_{0},\delta_{0})$. Utilizing \eqref{DM001} and \eqref{K001}--\eqref{AZ005}, we deduce that for $i=1,...,d-1$,
\begin{align}\label{QL001}
|F_{i}|\leq C\left(|y'|^{\gamma}\delta^{1-1/m}+\delta^{1+\gamma-1/m}\right),\quad\mathrm{in}\;B'_{R_{0}}.
\end{align}

We are now ready to prove Theorem \ref{thm001}.
\begin{proof}[Proof of Theorem \ref{thm001}]
Suppose that $u(0)=0$ and $\|u\|_{L^{\infty}(\Omega_{R_{0}})}=1$ without loss of generality. Let $v$ and $\bar{v}$ be defined by \eqref{ZKM001} and \eqref{KL01}--\eqref{ZK003}, respectively. First, we see from \eqref{AZ005}, \eqref{KL01} and \eqref{QL001} that for $\varepsilon\geq0$,
\begin{align*}
\|\nabla\bar{v}\|_{\varepsilon,-m\sigma_{0},1,B_{R_{0}}'}<\infty,\quad\|F\|_{\varepsilon,\gamma-m\sigma_{0},0,B_{R_{0}}'}<\infty,\quad\mathrm{with}\;\sigma_{0}=\frac{1}{m}.
\end{align*}

{\bf Case 1.} Consider the case when $\varepsilon=0$. Applying Proposition \ref{prop001} with $\sigma=\gamma-m\sigma_{0}$ and $\tau=m\sigma_{0}$, we obtain that for $0<\rho<R\leq R_{0}$,
\begin{align}\label{GQA001}
\left(\fint_{\partial B_{\rho}'}|\bar{v}(x')-\bar{v}(0')|^{2}\right)^{1/2}\leq C\rho^{\tilde{\alpha}},
\end{align}
where $\tilde{\alpha}=\min\{\alpha(\lambda_{1}),1+\gamma-m\sigma_{0}\}.$ In light of \eqref{DZNQ001}, it then follows from \eqref{GQA001} that for $x'_{0}\in B_{R_{0}/2}'\setminus\{0'\}$,
\begin{align}\label{LZMWN001A}
\int_{B'_{2c_{0}\delta^{1/m}_{0}}(x_{0}')}|\bar{v}-\bar{v}(0')|^{2}\leq \int_{B'_{|x_{0}'|+2c_{0}\delta^{1/m}_{0}}(0')}|\bar{v}-\bar{v}(0')|^{2}\leq C\delta^{\frac{2\tilde{\alpha}+d-1}{m}}_{0}.
\end{align}
From \eqref{AZ005}, we know
\begin{align}\label{RMBKL001}
|v(y',y_{d})-\bar{v}(y')|\leq2\delta_{0}\max_{y_{d}\in(-\delta_{0},\delta_{0})}|\partial_{d}v(y',y_{d})|\leq C\delta^{1-1/m},\quad\mathrm{in}\;Q_{R_{0},\delta_{0}}.
\end{align}
A consequence of \eqref{LZMWN001A}--\eqref{RMBKL001} shows that
\begin{align*}
&\fint_{Q_{2c_{0}\delta_{0}^{1/m},\delta_{0}}(x_{0}')}|v-\bar{v}(0')|^{2}dy\notag\\
&\leq\fint_{Q_{2c_{0}\delta_{0}^{1/m},\delta_{0}}(x_{0}')}2\big(|v-\bar{v}|^{2}+|\bar{v}-\bar{v}(0')|^{2})dy\leq C\delta_{0}^{\frac{2\tilde{\alpha}}{m}}.
\end{align*}
Let
\begin{align*}
\tilde{v}(y)=&v(\delta_{0}^{1/m}y'+x_{0}',\delta_{0}^{1/m}y_{d})-\bar{v}(0'),\notag\\
\tilde{b}_{ij}(y)=&\delta_{0}^{-1}b_{ij}(\delta_{0}^{1/m}y'+x_{0}',\delta_{0}^{1/m}y_{d}).
\end{align*}
Observe that $Q_{2c_{0}\delta_{0}^{1/m},\delta_{0}}(x_{0}')\subset Q_{R_{0},\delta_{0}}$ for $x_{0}'\in B_{R_{0}/2}'\setminus\{0'\}$, then $\tilde{v}$ satisfies
\begin{align}\label{EQM001}
\begin{cases}
-\partial_{i}(\tilde{b}_{ij}(y)\partial_{j}\tilde{v}(y))=0,&\mathrm{in}\; Q_{2c_{0},\delta_{0}^{1-1/m}},\\
\tilde{b}_{dj}(y)\partial_{j}\tilde{v}(y)=0,&\mathrm{on}\;\{y_{d}=\pm\delta_{0}^{1-1/m}\}.
\end{cases}
\end{align}
From \eqref{QWN001}, it follows that $\tilde{b}:=(\tilde{b}_{ij})$ verifies
\begin{align}\label{EQM002}
\frac{I}{C}\leq\tilde{b}\leq CI,\quad\mathrm{and}\;\|\tilde{b}\|_{C^{\mu}(Q_{2c_{0},\delta_{0}^{1-1/m}})}\leq C,\quad\mathrm{for}\;\mathrm{any}\;\mu\in(0,1].
\end{align}
For any $l\in\mathbb{N}$, define
\begin{align*}
S_{l}:=\{y\in\mathbb{R}^{d}\,|\,|y'|<2c_{0},\;(2l-1)\delta_{0}^{1-1/m}<y_{d}<(2l+1)\delta_{0}^{1-1/m}\},
\end{align*}
and
\begin{align*}
S:=\{y\in\mathbb{R}^{d}\,|\,|y'|<2c_{0},\,|y_{d}|<2c_{0}\}.
\end{align*}
Especially when $l=0$, $S_{0}=Q_{2c_{0},\delta_{0}^{1-1/m}}$. We first carry out even extension of $\tilde{v}$ in terms of $y_{d}=\delta_{0}^{1-1/m}$ and then perform the periodic extension with the period of $4\delta_{0}^{1-1/m}$. That is, we obtain
\begin{align*}
\hat{v}(y):=\tilde{v}(y',(-1)^{l}(y_{d}-2l\delta_{0}^{1-1/m})),\quad\mathrm{in}\;S_{l},\,l\in\mathbb{Z},
\end{align*}
For $k=1,...,d-1$ and any $l\in\mathbb{Z}$, the corresponding coefficients turns into
\begin{align*}
\hat{a}_{dk}(y)=\hat{a}_{kd}(y):=(-1)^{l}\tilde{a}_{kd}(y',(-1)^{l}(y_{d}-2l\delta_{0}^{1-1/m})),\quad\mathrm{in}\;S_{l},
\end{align*}
while, for other indices,
\begin{align*}
\hat{a}_{ij}(y):=\tilde{a}_{ij}(y',(-1)^{l}(y_{d}-2l\delta_{0}^{1-1/m})),\quad\mathrm{in}\;S_{l}.
\end{align*}
Hence, $\hat{v}$ and $\hat{a}_{ij}$ are defined in $Q_{2,\infty}$. Moreover, $\hat{v}$ solves
\begin{align*}
\partial_{i}(\hat{a}_{ij}\partial_{j}\hat{v})=0,\quad\mathrm{in}\;S,
\end{align*}
by utilizing the conormal boundary conditions. It then follows from Proposition 4.1 of \cite{LN2003} and Lemma 2.1 of \cite{LY202102} that
\begin{align*}
\|\nabla\hat{v}\|_{L^{\infty}(\frac{1}{2}S)}\leq C\|\hat{v}\|_{L^{2}(S)}\leq C\delta_{0}^{\frac{\tilde{\alpha}}{m}}.
\end{align*}
Rescaling back to $u$, we deduce that for $x_{0}=(x_{0}',x_{d})\in\Omega_{R_{0}/2}\setminus\{x_{0}'=0'\}$,
\begin{align}\label{GNRQ0001}
|\nabla u(x_{0})|\leq\|\nabla u\|_{L^{\infty}(\Omega_{c_{0}\delta^{1/m}}(x_{0}'))}\leq C\delta_{0}^{\frac{\tilde{\alpha}-1}{m}}\leq C|x_{0}'|^{\tilde{\alpha}-1},
\end{align}
where $\tilde{\alpha}=\min\{\alpha(\lambda_{1}),1+\gamma-m\sigma_{0}\}$.

In light of \eqref{GNRQ0001}, the upper bound $|\nabla u(x)|\leq C|x'|^{-m\sigma_{0}}$ has been improved to be $|\nabla u(x)|\leq C|x'|^{\tilde{\alpha}-1}$, where $\tilde{\alpha}-1=\min\{\alpha(\lambda_{1})-1,\gamma-m\sigma_{0}\}.$ If $1+\gamma-m\sigma_{0}>\alpha$, the proof is complete. Otherwise, if $1+\gamma-m\sigma_{0}<\alpha$, then choose $\sigma_{1}=\sigma_{0}-\frac{\gamma}{m}$ and repeat the above-mentioned argument. For the purpose of letting $\alpha(\lambda_{1})-1\neq-m\sigma_{0}+k\gamma$ for any $k\geq1$, we may decrease $\gamma$ if necessary. By repeating the argument above with finite times, we prove that \eqref{ma001} holds.

{\bf Case 2.} Consider the case when $\varepsilon>0$. Denote
\begin{align*}
\omega(\rho):=\left(\fint_{B'_{(1\pm\bar{c}_{0})\rho}}\Big|\bar{v}(x')-(\bar{v})^{\kappa}_{B'_{(1\pm\bar{c}_{0})\rho}}\Big|^{2}\right)^{1/2},\quad\text{for }0<\rho<(1-\bar{c}_{0})^{2}R_{0}.
\end{align*}
Picking $\sigma=\gamma-m\sigma_{0}$ and $\tau=m\sigma_{0}$ in Proposition \ref{prop003}, we have
\begin{align*}
\omega(\rho)\leq C\Big(\frac{\rho}{R}\Big)^{\alpha(\lambda_{1})}\omega(R)+C\Big(\frac{R}{\rho}\Big)^{\frac{d+m-2}{2}}R^{1-m\sigma_{0}}\left[R^{\gamma}\left(\frac{\sqrt{\varepsilon}}{R^{m/2}}+1\right)+\left(\frac{\varepsilon}{R^{m}}\right)^{\beta(\lambda_{1})}\right],
\end{align*}
where $\beta(\lambda_{1})>\alpha(\lambda_{1})$ is defined in \eqref{beta001}. In the case of $\varepsilon>0$, we pick a small positive constant $\bar{\mu}$ such that $\bar{\mu}\tilde{\alpha}(\lambda_{1})\leq\gamma$. Then for any $\varepsilon^{\frac{1}{m+\mu}}\leq\rho<(1-\bar{c}_{0})^{2}R\leq (1-\bar{c}_{0})^{2}R_{0}$,
\begin{align*}
\omega(\rho)\leq C\Big(\frac{\rho}{R}\Big)^{\alpha(\lambda_{1})}\omega(R)+C\Big(\frac{R}{\rho}\Big)^{\frac{d+m-2}{2}}R^{1-m\sigma_{0}+\bar{\mu}\beta(\lambda_{1})}.
\end{align*}
Making use of Lemma 5.13 in \cite{GM2012}, we obtain that for $\varepsilon^{\frac{1}{m+\mu}}\leq\rho<(1-\bar{c}_{0})^{2}R_{0}$,
\begin{align*}
\omega(\rho)\leq C\rho^{\hat{\alpha}},\quad\text{with $\hat{\alpha}:=\min\{\alpha(\lambda_{1}),1-m\sigma_{0}+\mu\beta(\lambda_{1})\}$.}
\end{align*}
Then for $x'_{0}\in B'_{(1-\bar{c}_{0})^{2}R_{0}}\setminus B'_{2\varepsilon^{1/(m+\mu)}}$,
\begin{align*}
&\int_{B'_{2c_{0}\delta^{1/m}_{0}}(x_{0}')}\Big|\bar{v}-(\bar{v})^{\kappa}_{B'_{(1\pm\bar{c}_{0})|x_{0}'|}}\Big|^{2}\notag\\
&\leq \int_{B'_{(1\pm\bar{c}_{0})|x_{0}'|}(0')}\Big|\bar{v}-(\bar{v})^{\kappa}_{B'_{(1\pm\bar{c}_{0})|x_{0}'|}}\Big|^{2}\leq C\delta^{\frac{2\hat{\alpha}+d-1}{m}}_{0}.
\end{align*}
This, together with \eqref{RMBKL001}, shows that
\begin{align*}
&\fint_{Q_{2c_{0}\delta_{0}^{1/m},\delta_{0}}(x_{0}')}\Big|v-(\bar{v})^{\kappa}_{B'_{(1\pm\bar{c}_{0})|x_{0}'|}}\Big|^{2}\leq C\delta_{0}^{\frac{2\hat{\alpha}}{m}}.
\end{align*}
Similarly as before, define
\begin{align*}
\tilde{v}(y)=&v(\delta_{0}^{1/m}y'+x_{0}',\delta_{0}^{1/m}y_{d})-(\bar{v})^{\kappa}_{B'_{(1\pm\bar{c}_{0})|x_{0}'|}},\notag\\
\tilde{b}_{ij}(y)=&\delta_{0}^{-1}b_{ij}(\delta_{0}^{1/m}y'+x_{0}',\delta_{0}^{1/m}y_{d}).
\end{align*}
Therefore, $\tilde{v}(y)$ verifies equation \eqref{EQM001} with $\tilde{b}:=(\tilde{b}_{ij})$ satisfying \eqref{EQM002}. Then using the same ``flipping argument" as above, we obtain
\begin{align}\label{GNRQ0001001}
|\nabla u(x_{0})|\leq C\delta_{0}^{\frac{\hat{\alpha}-1}{m}},\quad\mathrm{for}\;x'_{0}\in B'_{(1-\bar{c}_{0})^{2}R_{0}}\setminus B'_{2\varepsilon^{1/(m+\mu)}}.
\end{align}
From \eqref{GNRQ0001001}, we have
\begin{align*}
\mathrm{osc}_{\Omega_{2\rho}\setminus\Omega_{\rho}}u\leq C\rho^{\hat{\alpha}},\quad\text{for any $2\varepsilon^{\frac{1}{m+\mu}}\leq\rho<R_{0}/2$},
\end{align*}
This, together with the maximum principle, reads that for $|x_{0}'|\leq2\varepsilon^{\frac{1}{m+\mu}},$
\begin{align*}
\mathrm{osc}_{\Omega_{2c_{0}\delta_{0}^{1/m}}(x_{0})}u\leq\mathrm{osc}_{\Omega_{4\varepsilon^{\frac{1}{m+\mu}}}}u\leq C\varepsilon^{\frac{\hat{\alpha}}{m+\mu}}.
\end{align*}
Then for any $x_{0}\in\Omega_{2\varepsilon^{\frac{1}{m+\mu}}},$ we have
\begin{align}\label{WNZL001}
\|u-u(x_{0})\|_{L^{\infty}(\Omega_{2c_{0}\delta_{0}^{1/m}}(x_{0}))}\leq C\varepsilon^{\frac{\hat{\alpha}}{m+\mu}}.
\end{align}
Hence, applying the changes of variables in \eqref{change001} for $u-u(x_{0})$, it follows from the same ``flipping argument" and \eqref{WNZL001} that
\begin{align*}
|\nabla u(x_{0})|\leq C\delta_{0}^{-1/m}\varepsilon^{\frac{\hat{\alpha}}{m+\mu}},\quad\text{for $x_{0}\in\Omega_{2\varepsilon^{\frac{1}{m+\mu}}}$}.
\end{align*}
This, in combination with \eqref{GNRQ0001001}, shows that
\begin{align*}
|\nabla u(x)|\leq C\delta^{-\frac{1}{m}+\frac{\hat{\alpha}}{m+\mu}},\quad\text{for any $x\in\Omega_{(1-\bar{c}_{0})^{2}R_{0}}$}.
\end{align*}
Therefore, we improve the aforementioned upper bound $|\nabla u(x)|\leq C(\varepsilon+|x'|^{m})^{-\sigma_{0}}$ to be $|\nabla u(x)|\leq C(\varepsilon+|x'|^{m})^{-\sigma_{0}+\frac{\hat{\alpha}}{m+\mu}}$ with $\hat{\alpha}(\lambda_{1})=\min\{\alpha(\lambda_{1}),1-m\sigma_{0}+\mu\beta(\lambda_{1})\}.$ If $1-m\sigma_{0}+\mu\beta(\lambda_{1})\geq\alpha(\lambda_{1})$, then we have $|\nabla u|\leq C\delta^{-\frac{1}{m}+\frac{\alpha(\lambda_{1})}{m+\mu}}$ in $\Omega_{R_{0}/2}$. By letting $\mu\rightarrow0$, we complete the proof. Otherwise, if $1-m\sigma_{0}+\mu\beta(\lambda_{1})<\alpha(\lambda_{1})$, then set $\sigma_{1}=\sigma_{0}-\frac{1-m\sigma_{0}+\mu\beta(\lambda_{1})}{m+\mu}$ and repeat the above-mentioned argument with
\begin{align*}
\|\nabla\bar{v}\|_{\varepsilon,-m\sigma_{1},1,B_{R_{0}}'}+\|F\|_{\varepsilon,\gamma-m\sigma_{1},0,B_{R_{0}}'}<\infty.
\end{align*}
Denote
\begin{align*}
\sigma_{k}=\sigma_{k-1}-\frac{1-m\sigma_{k-1}+\mu\beta(\lambda_{1})}{m+\mu},\quad k\geq1.
\end{align*}
Then we have
\begin{align*}
\sigma_{k}-\frac{1}{m}=\frac{2m+\mu}{m+\mu}(\sigma_{k-1}-\frac{1}{m})-\frac{\mu\beta(\lambda_{1})}{m+\mu}.
\end{align*}
Note that $\sigma_{0}=\frac{1}{m}$, it then follows that
\begin{align*}
\sigma_{k}=\sigma_{0}-\frac{\sum^{k-1}_{i=0}(\frac{2m+\mu}{m+\mu})^{i}\mu\beta(\lambda_{1})}{m+\mu},
\end{align*}
which yields that
\begin{align*}
1-m\sigma_{k}+\mu\beta(\lambda_{1})=\left(\sum^{k-1}_{i=0}\Big(\frac{2m+\mu}{m+\mu}\Big)^{i}\frac{ m}{m+\mu}+1\right)\mu\beta(\lambda_{1})\rightarrow+\infty,\;\,\text{as $k\rightarrow+\infty$}.
\end{align*}
Then there exists some $k_{0}\in\mathbb{N}$ such that
\begin{align*}
\alpha(\lambda_{1})\leq\left(\sum^{k_{0}-1}_{i=0}\Big(\frac{2m+\mu}{m+\mu}\Big)^{i}\frac{ m}{m+\mu}+1\right)\mu\beta(\lambda_{1})=1-m\sigma_{k_{0}}+\mu\beta(\lambda_{1}),
\end{align*}
and
\begin{align*}
\alpha(\lambda_{1})>\left(\sum^{k_{0}-2}_{i=0}\Big(\frac{2m+\mu}{m+\mu}\Big)^{i}\frac{ m}{m+\mu}+1\right)\mu\beta(\lambda_{1}),\quad\text{if $k_{0}\geq2$.}
\end{align*}
Therefore, after $k_{0}$ iterations, we get
\begin{align*}
|\nabla u(x)|\leq C\delta^{-\frac{1}{m}+\frac{\alpha(\lambda_{1})}{m+\mu}},\quad\text{in $\Omega_{R_{0}/4}$},
\end{align*}
where we used the fact that $(1-\bar{c}_{0})^{2}\geq\frac{1}{4}$. Sending $\mu\rightarrow0$, we obtain that \eqref{U002} holds.

\end{proof}

\section{The proof of Theorem \ref{thm002}}\label{SEC03}

Before giving the proof of Theorem \ref{thm002}, we first demonstrate the validity of the assumed condition ``the eigenspace corresponding to $\lambda_{1}$ contains a function which is odd with respect to some $x_{j_{0}}$, $j_{0}\in\{1,...,d-1\}$."
\begin{definition}
If a function space on $\mathbb{S}^{d-2}\subset\mathbb{R}^{d-1}$ is spanned by functions which are odd with respect to some $x_{i}$ and even in other variables, we say that it satisfies the property $O$.
\end{definition}
For $\mu\in\mathbb{R}$ and $d\geq3$, consider
\begin{align*}
L_{\mu}=-\mathrm{div}_{\mathbb{S}^{d-2}}((1+\mu b(x))\nabla_{\mathbb{S}^{d-2}}),\quad\mathrm{on}\;\mathbb{S}^{d-2},
\end{align*}
where $b(x)\in L^{\infty}(\mathbb{S}^{d-2}).$ Denote by $\lambda_{1,\mu}$ and $V_{1,\mu}$ the corresponding first nonzero eigenvalue and the eigenspace of the following problem
\begin{align}\label{KWZN001}
L_{\mu}u=\lambda(1+\mu b(x))u.
\end{align}
Recall Proposition 5.6 in \cite{DLY2022} as follows:
\begin{lemma}[Proposition 5.6 of \cite{DLY2022}]\label{lemmaAZ001}
Consider the eigenvalue problem \eqref{KWZN001}. Let $b(x)$ be even in all variables and $V_{1,\mu_{0}}$ satisfy the property $O$ for some $\mu_{0}\in\mathbb{R}$. Then there exists a small positive constant $\varepsilon_{0}:=\varepsilon_{0}(d,\|b\|_{L^{\infty}},\mu_{0})$ such that $V_{1,\mu}$ verifies the property $O$ for every $\mu\in(\mu_{0}-\varepsilon_{0},\mu_{0}+\varepsilon_{0}).$
\end{lemma}
Then applying Lemma \ref{lemmaAZ001} with $d\geq3$, $\|b\|_{L^{\infty}}\leq2$ and $\mu_{0}=0$, we derive a small constant $\varepsilon_{0}=\varepsilon_{0}(d)$ such that $V_{1,\mu}$ verifies the property $O$ in $(-\varepsilon_{0},\varepsilon_{0})$. Note that by a rotation of the coordinates if necessary, we have $\kappa_{i}\geq\kappa_{i+1}$ for $i=1,...,d-2.$ Choose $\mu=\frac{\varepsilon_{0}}{2}$ and
\begin{align}\label{DZNC001}
b(x)=\frac{2\Big[\sum^{d-2}_{i=1}\kappa_{i}|x_{i}|^{m}-\Big(1-\big(1-\sum^{d-2}_{i=1}x_{i}^{2}\big)^{m/2}\Big)\kappa_{d-1}\Big]}{\varepsilon_{0}\kappa_{d-1}},
\end{align}
which implies that $1+\mu b=\frac{1}{\kappa_{d-1}}\sum^{d-1}_{i=1}\kappa_{i}|x_{i}|^{m}$. Therefore, if we assume that $\frac{\kappa_{1}}{\kappa_{d-1}}-1\leq\varepsilon_{0}$, then $b(x)$ defined in \eqref{DZNC001} satisfies $\|b\|_{L^{\infty}}\leq2$. In fact, since
\begin{align*}
1-\left(1-\sum^{d-2}_{i=1}x_{i}^{2}\right)^{m/2}\geq&\left(\sum^{d-2}_{i=1}x_{i}^{2}\right)^{m/2}\geq\sum^{d-2}_{i=1}|x_{i}|^{m},
\end{align*}
then
\begin{align*}
|b|\leq \frac{2\sum^{d-2}_{i=1}(\kappa_{i}-\kappa_{d-1})|x_{i}|^{m}}{\varepsilon_{0}\kappa_{d-1}}\leq\frac{2}{\varepsilon_{0}}\left(\frac{\kappa_{1}}{\kappa_{d-1}}-1\right)\leq2.
\end{align*}
Consequently, under the condition of $\frac{\kappa_{1}}{\kappa_{d-1}}-1\leq\varepsilon_{0}$, we obtain that the eigenspace corresponding to $\lambda_{1}$ of the eigenvalue problem \eqref{eigen001} verifies the property $O$.

We now give the proof of Theorem \ref{thm002}.
\begin{proof}[Proof of Theorem \ref{thm002}]
For simplicity, write $\alpha:=\alpha(\lambda_{1})$ in the following. Assume without loss of generality that $R_{0}=1$. Denote
\begin{align*}
\bar{u}(x'):=\fint^{h_{1}(x')}_{h_{2}(x')}u(x',x_{d})\,dx_{d},\quad\mathrm{in}\;\Omega_{1}\setminus\{x'=0'\}.
\end{align*}
Then $\bar{u}$ solves
\begin{align*}
\mathrm{div}\Big[\Big(\sum^{d-1}_{i=1}\kappa_{i}|x_{i}|^{m}\Big)\nabla\bar{u}\Big]=\mathrm{div}F,\quad\mathrm{in}\; B_{1}',
\end{align*}
where $F=(F_{1},...,F_{d-1})$, $F_{i}=-\overline{b^{i}\partial_{d}u}-e\partial_{i}\bar{u}$ with $e=O(|x'|^{m+\gamma})$ and
\begin{align*}
b^{i}(x)=(h_{1}(x')-h_{2}(x'))\partial_{i}h_{2}(x')+(x_{d}-h_{2}(x'))\partial_{i}(h_{1}(x')-h_{2}(x')).
\end{align*}
Then we have from $\mathrm{(}${\bf{H1}}$\mathrm{)}$--$\mathrm{(}${\bf{H2}}$\mathrm{)}$ and Theorem \ref{thm001} that
\begin{align}\label{MWALZ001}
|F(x')|\leq C|x'|^{m-1+\gamma+\alpha},\quad\mathrm{in}\;B'_{1}\setminus\{0'\}.
\end{align}

{\bf Step 1.} Denote by $Y_{k,l}$ the normalized eigenfunction corresponding to $\lambda_{k}$ which is $(k+1)$-th eigenvalue of problem \eqref{eigen001}. Then $\{Y_{k,l}\}$ is an orthonormal basis of $L^{2}(\mathbb{S}^{d-2})$ with the inner product \eqref{inner001}. From the assumed condition, we denote by $Y_{1,j_{0}}$ the eigenfunction which is odd with respect to $x_{j_{0}}$. Then the eigenfunction $Y_{1,j_{0}}$ corresponding to $\lambda_{1}$ on the half sphere $\mathbb{S}^{d-2}\cap\{x_{j_{0}}>0\}$ satisfies the zero Dirichlet boundary condition. $\lambda_{1}$, as the first nonzero eigenvalue of problem \eqref{eigen001} on the sphere, is also the same on the half sphere. Then $\lambda_{1}$ is simple and $Y_{1,j_{0}}$ keeps the same sign on the half sphere. So we assume without loss of generality that $Y_{1,j_{0}}>0$ in $\{x_{j_{0}}>0\}$ and  $Y_{1,j_{0}}<0$ in $\{x_{j_{0}}<0\}$. Based on the fact that the domain $\Omega:=D\setminus\overline{D_{1}\cup D_{2}}$ is symmetric with respect to each $x_{i}$, $1\leq i\leq d-1$, and $\varphi$ is an odd function of $x_{j_{0}}$, we deduce from the classical elliptic theory that $u$ is an odd function of $x_{j_{0}}$ and thus $\bar{u}$ is also odd in $x_{j_{0}}$, and $\bar{u}(0')=0$. Then we utilize the basis $\{Y_{k,l}\}$ to carry out the decomposition as follows:
\begin{align}\label{WHDZ001}
\bar{u}(x')=\sum^{\infty}_{k=1}\sum^{N(k)}_{l=1}U_{k,l}(r)Y_{k,j}(\xi),\quad\mathrm{in}\;B_{R_{0}}'\setminus\{0'\},
\end{align}
where $U_{k,l}\in C([0,1))\cap C^{\infty}((0,1))$ is given by $U_{k,l}(r)=\fint_{\mathbb{S}^{d-2}}\kappa(\xi)\bar{u}(r,\xi)Y_{k,l}(\xi)d\xi$. Based on these facts above, we deduce that $U_{1,j_{0}}(0)=0$,
\begin{align*}
LU_{1,j_{0}}:=U_{1,j_{0}}''+\frac{m+d-2}{r}U_{1,j_{0}}'-\frac{\lambda_{1}}{r^{2}}U_{1,j_{0}}=H(r),\quad0<r<1,
\end{align*}
where
\begin{align*}
H(r)&=\int_{\mathbb{S}^{d-2}}\frac{(\mathrm{div}F)Y_{1,j_{0}}(\xi)}{\kappa(\xi)r^{m}}d\xi=\int_{\mathbb{S}^{d-2}}\frac{\partial_{r}F_{r}+\frac{1}{r}\nabla_{\xi}F_{\xi}}{\kappa(\xi)r^{m}}Y_{1,j_{0}}(\xi)d\xi\notag\\
&=\partial_{r}\left(\int_{\mathbb{S}^{d-2}}\frac{F_{r}Y_{1,j_{0}}}{\kappa(\xi)r^{m}}d\xi\right)+\int_{\mathbb{S}^{d-2}}\left(\frac{mF_{r}Y_{1,j_{0}}}{\kappa(\xi)r^{m+1}}-\frac{F_{\xi}}{r^{m+1}}\nabla_{\xi}\Big(\frac{Y_{1,j_{0}}}{\kappa(\xi)}\Big)\right)d\xi\notag\\
&=:\partial_{r}A(r)+B(r),\quad\mathrm{in}\;(0,1),
\end{align*}
with $A(r),B(r)\in C^{1}([0,1))$. Making use of \eqref{MWALZ001}, we obtain
\begin{align}\label{AB001}
|A(r)|\leq C(d)r^{\gamma+\alpha-1},\quad\mathrm{and}\;|B(r)|\leq C(d)r^{\gamma+\alpha-2},\quad\mathrm{in}\;(0,1).
\end{align}

\noindent{\bf Step 2.} We first demonstrate that $U_{1,j_{0}}$ can be decomposed as follows:
\begin{align}\label{TMZW001}
U_{1,j_{0}}(r)=C_{1}r^{\alpha}+O(r^{\alpha+\gamma}),
\end{align}
where $\alpha=\alpha(\lambda_{1})$. Observe that $g=r^{\alpha}$ satisfies $Lg=0$. Define
\begin{align*}
w(r):=\int_{0}^{r}\frac{1}{s^{d+m+2\alpha-2}}\int^{s}_{0}t^{d+m+\alpha-2}H(t)dtds,\quad\mathrm{in}\;(0,1).
\end{align*}
Denote $v=gw$. Then $v$ verifies
\begin{align*}
Lv=\frac{g}{G}(Gw')'=H,\quad\mathrm{with}\; G=g^{2}r^{d+m-2}.
\end{align*}
From \eqref{AB001}, we have $|w(r)|\leq Cr^{\gamma}$ and thus $|v(r)|\leq Cr^{\alpha+\gamma}$. Due to the fact that $U_{1,j_{0}}-v$ stays bounded and verifies $L(U_{1,j_{0}}-v)=0$ in $(0,1)$, we obtain that $U_{1,j_{0}}=C_{1}g+v$ and thus \eqref{TMZW001} holds.

Claim that the constant $C_{1}>0$. By the symmetry and convexity of the domain, we obtain that $\partial_{\nu}x_{j_{0}}\geq0$ in $\{x_{j_{0}}\geq0\}$ and $\partial_{\nu}x_{j_{0}}\leq0$ in $\{x_{j}\leq0\}$. Then combining the assumed conditions in Theorem \ref{thm002}, $x_{j_{0}}$ becomes a subsolution of \eqref{con002} in $\{x_{j_{0}}\geq0\}$, while it is a supersolution of \eqref{con002} in $\{x_{j_{0}}\leq0\}$. This implies that $|u(x)|\geq|x_{j_{0}}|$ in $\Omega$ and then $|\bar{u}(x)|\geq|x_{j_{0}}|$ in $B_{1}'$. In light of the fact that $Y_{1,j_{0}}$ and $x_{j_{0}}$ possess the same sign, we obtain
\begin{align*}
U_{1,j_{0}}=\fint_{\mathbb{S}^{d-2}}\kappa(\xi)\bar{u}(r,\xi)Y_{1,j_{0}}(\xi)d\xi\geq Cr,\quad\text{for some constant }C>0,
\end{align*}
which, together with \eqref{TMZW001} and the assumed condition that $\gamma>1-\alpha$, reads that $C_{1}>0$.

From \eqref{WHDZ001} and \eqref{TMZW001}, we deduce
\begin{align*}
\left(\int_{\mathbb{S}^{d-2}}\kappa(\xi)|\bar{u}(r,\xi)|^{2}d\xi\right)^{1/2}\geq|U_{1,j_{0}}|\geq\frac{C_{1}}{2}r^{\alpha},\quad\mathrm{in}\;(0,r_{0}),
\end{align*}
for some positive constant $r_{0}$. Then there exists a $\xi_{0}(r)\in\mathbb{S}^{d-2}$ for any $r\in(0,r_{0})$ such that
\begin{align}\label{MAZE001}
|\bar{u}(r,\xi_{0}(r))|\geq\frac{1}{C_{2}}r^{\alpha},\quad\text{for some constant $C_{2}>0$.}
\end{align}
Using \eqref{ma001}, we have
\begin{align}\label{MAZE002}
|u(r,\xi_{0}(r))-\bar{u}(r,\xi_{0}(r))|\leq Cr^{m}\sup\limits_{h_{2}(x')\leq x_{d}\leq h_{1}(x')}|\partial_{d}u(r,\xi_{0}(r),x_{d})|\leq Cr^{m+\alpha-1}.
\end{align}
From \eqref{MAZE001} and \eqref{MAZE002}, we deduce
\begin{align*}
|u(r,\xi_{0}(r),0)|\geq\frac{1}{2C_{2}}r^{\alpha},\quad\mathrm{in}\;(0,r_{1}),
\end{align*}
for some positive constant $r_{1}$. Write $x_{0}=(r,\xi_{0}(r),0)$. Utilize \eqref{ma001} and pick a sufficiently large $x_{0}$-independent constant $C_{3}$ such that
\begin{align*}
|u(x_{0}C_{3}^{-1})|\leq C|x_{0}'|^{\alpha}C_{3}^{-\alpha}\leq\frac{1}{4C_{2}}|x_{0}'|^{\alpha}.
\end{align*}
Then there exists a point $x$ in the line segment between $x_{0}$ and $x_{0}/C_{3}$ such that
\begin{align*}
|\nabla u(x)|\geq\frac{1}{C}|x'|^{\alpha-1},
\end{align*}
where $C$ is some positive constant, which may depend on $d,m,\kappa_{i}$, $i=1,...,d-1$, and the upper bounds of $\|\partial D_{j}\|_{C^{4}}$, $j=1,2.$ The proof is complete.

\end{proof}

\noindent{\bf{\large Acknowledgements.}}
The author would like to thank Prof. C.X. Miao for his constant encouragement and useful discussions. The author was partially supported by CPSF (2021M700358).

\bibliographystyle{plain}

\def\cprime{$'$}

\end{document}